\documentclass[12pt]{amsart}
\usepackage[utf8]{inputenc}
\usepackage{accents}
\usepackage{amsthm}
\usepackage{amsmath}
\usepackage{stmaryrd}
\usepackage{amssymb} 
\usepackage{comment}
\usepackage{graphicx}
\usepackage{mathrsfs}
\usepackage{fullpage}
\usepackage{mathtools}
\usepackage{colonequals}
\usepackage{hyperref}
\usepackage[all]{xy}
\usepackage[usenames,dvipsnames]{xcolor}
\usepackage[ left = 1.2in, right = 1.2in, 
             top = 1.2in, bottom = 1.2in ]{geometry}
\usepackage{tikz-cd}
\usetikzlibrary{calc}

\numberwithin{equation}{section} 
\numberwithin{figure}{section} 

\newtheorem{theorem}[equation]{Theorem}
\newtheorem{corollary}[equation]{Corollary}
\newtheorem{lemma}[equation]{Lemma}
\newtheorem{prop}[equation]{Proposition}

\theoremstyle{definition}
\newtheorem{remark}[equation]{Remark}
\newtheorem{example}[equation]{Example}

\theoremstyle{definition}
\newtheorem{definition}[equation]{Definition}

\DeclareMathOperator{\rank}{rank}
\DeclareMathOperator{\Span}{Span}
\DeclareMathOperator{\Tr}{Tr}
\DeclareMathOperator{\IG}{IG}

\DeclareMathOperator{\Num}{Num}
\DeclareMathOperator{\Vol}{Vol}
\DeclareMathOperator{\length}{length}
\newcommand{\C}{\mathbb{C}} 
\newcommand{\F}{\mathbb{F}}
\newcommand{\Q}{\mathbb{Q}}
\newcommand{\R}{\mathbb{R}}
\newcommand{\<}{\langle}
\renewcommand{\>}{\rangle}
\newcommand{\e}{\mathbf{e}}
\renewcommand{\t}{\mathbf{t}}
\newcommand{\uu}{\mathbf{u}}
\renewcommand{\v}{\mathbf{v}}
\newcommand{\w}{\mathbf{w}}
\newcommand{\x}{\mathbf{x}}
\newcommand{\y}{\mathbf{y}}
\newcommand{\cF}{{\mathcal{F}}}
\newcommand{\cH}{{\mathcal{H}}}
\newcommand{\cK}{{\mathcal{K}}}
\newcommand{\cP}{{\mathcal{P}}}
\newcommand{\cL}{{\mathcal{L}}}
\newcommand{\Z}{{\mathbb{Z}}}

\newcommand{\drawSquareFrom}[7]{%
    \coordinate (A) at (#1,#2);           
    \coordinate (B) at ($(A)+(1,0)$);     
    \coordinate (C) at ($(A)+(1,1)$);     
    \coordinate (D) at ($(A)+(0,1)$);     

    \draw[#3, thick] (A) -- (B);  
    \draw[#4, thick] (B) -- (C);  
    \draw[#5, thick] (C) -- (D);  
    \draw[#6, thick] (D) -- (A);  
    \draw[#7, thick] (A) -- (C);  
}

\newcommand{\orangeline}[2]{%
    \coordinate (A) at (#1+.5,4.2);           
    \coordinate (B) at (#2+.5,5.8);           
    \draw[orange, thick] (A) -- (B);  
}

\newcommand{\loworangeline}[1]{
    \coordinate (A) at (5.5,1.2);           
    \coordinate (B) at (#1+.5,2.8);       
    \draw[orange, thick] (A) -- (B); 
}

\newcommand{\highorangeline}[1]{
    \coordinate (A) at (#1+.5,7.2);       
    \coordinate (B) at (5.5,8.8);           
    \draw[orange, thick] (A) -- (B); 
}

\begin{document}

\title{Volume Polynomials and Log-concavity of the Characteristic Polynomials of Matroids}
\markright{Volume Polynomials and Characteristic Polynomials}
\author{Eric Katz}

\begin{abstract}
The Rota--Heron--Welsh conjecture (now a theorem of Adiprasito, Huh, and the author) asserts the log-concavity of the characteristic polynomial of matroids. 
We give an exposition of the Lorentzian polynomial proof following the work of Br\"{a}nd\'{e}n and Leake aimed at undergraduate and beginning graduate students.
\end{abstract}

\maketitle

\noindent

\section{Introduction}

Given the exciting developments in matroid theory, one might ask ``Could a proof of the Rota--Heron--Welsh conjecture be taught to undergraduates?'' --- and be too preoccupied to stop and think if one should. Still, the question is interesting, and this manuscript is an attempt at an answer.

The Rota--Heron--Welsh conjecture, resolved by Adiprasito, Huh, and the author \cite{AHK} states the characteristic polynomial of a matroid is log-concave, that is, its coefficients satisfy a certain system of inequalities. It was stated as such by Rota at the time of his ICM address \cite{Rota:ICM} following work of Read, Hoggar, Heron, and Welsh. Log-concavity often reflects the structure of zeroes of a polynomial. Since the characteristic polynomial is a generalization of the chromatic polynomial of graphs, it is natural to care about its zeroes. Little progress was made on the Rota--Heron--Welsh conjecture until Huh's revolutionary work \cite{Huh:chromatic} which recognized the coefficients of the characteristic polynomial for realizable matroids over $\C$ as invariants of a certain complex singularity satisfying a log-concavity result. Later work of Huh and the author \cite{HuhKatz} reformulated this approach using tropical geometry and extended it to all realizable matroids by employing the Khovanskii--Teissier inequality. The conjecture was resolved with Adiprasito through the development of combinatorial Hodge theory on the matroid Chow ring $A^*(M)$.

Still, it was thought that the proof in \cite{AHK} was overkill: it establishes a suite of positivity properties called the K\"{a}hler package on $A^*(M)$ where one only needs it on degree $1$ classes, i.e., those in $A^1(M)$. This amounts to showing that a certain matrix has a unique positive eigenvalue (i.e., a Lorentzian signature). This philosophy became the cornerstone for Lorentzian polynomials as developed by Br\"{a}nd\'{e}n and Huh \cite{BH:Lorentzian}. In some sense, Lorentzian polynomials rely on a bootstrap similar to but simpler than the K\"{a}hler bootstrap in \cite{AHK} that gives the Lorentzian signature condition.
The work of  Br\"{a}nd\'{e}n--Huh proved the ultra log-concavity of independence numbers of matroids, also independently resolved with similar techniques by Anari, Liu, Gharan, and Vinzant. Backman, Eur, and Simpson \cite{BES} were the first to give a Lorentzian proof of the Rota--Heron--Welsh conjecture. They identified a simplicial cone $\cK_M^\nabla$ in $A^1(M)$ with positivity properties and used it to prove that the volume polynomial was Lorentzian. A later proof by Br\"{a}nd\'{e}n and Leake \cite{BL:RHW} enlarged the purview of Lorentzian polynomials to more general cones and made use of the larger ample cone $\cK$  (see \cite{BL:oncones} for this proof's incorporation in the theory of Lorentzian polynomials on cones). 

We provide an exposition of the Lorentzian polynomial proof together with the necessary background on chromatic polynomials, matroids, and linear algebra assuming only material taught in the standard undergraduate curriculum. 

We outline the contents of this paper and the strategy of the proof. In Section~\ref{s:chromatic}, we review the chromatic polynomial of graphs, that is, to a graph $\Gamma$, there is a polynomial $P_{\Gamma}(q)$ such that for a positive integer $n$, $P_{\Gamma}(n)$ are the numbers of proper colorings of $\Gamma$ with $n$ colors. We reformulate the construction of the chromatic polynomial in terms of lattice of flats, a certain partially ordered set of subgraphs of $\Gamma$. In Section~\ref{s:matroids}, we generalize from graphs to matroids, introducing the lattice of flats, $\cL(M)$, and the rank function on it. Chromatic polynomials are generalized to characteristic polynomials in Section~\ref{s:characteristic}: to a matroid $M$, we attach the characteristic and reduced characteristic polynomial. They are written, respectively, as
\begin{align*}
  \chi_M(q)&=\mu_0q^{r+1}-\mu_1q^r+\dots+(-1)^{r+1}\mu_{r+1}q^0,\\
  \overline{\chi}_M(q)&=\mu^0q^r-\mu^1q^{r-1}+\dots+(-1)^r\mu^r.
\end{align*}
The Rota--Heron--Welsh conjecture is
\begin{theorem}
For $1\leq i\leq r$, $\mu_{i-1}\mu_{i+1}\leq \mu_i^2.$
\end{theorem}
In Section~\ref{s:why}, we digress to explain why one might care about log-concavity. A critical tool, the volume polynomial $V_M$ of a matroid $M$ is introduced in Section~\ref{s:volume}. 
It is defined by differential equations as a function $V_M\colon L_M\to \R$ where $L_M$ is a certain vector space attached to $M$. We relate the reduced characteristic polynomial to $V_M$ through a new proof of a lemma by the author with Huh \cite{HuhKatz}:
\begin{lemma}
    There is an equality $D_\alpha^{r-k} D_\beta^k V_M=\mu^k(M)$.
\end{lemma}
Here $\alpha$ and $\beta$ are particular elements of $L_M$ (see \cite{Ardila:intersectiontheoryonmatroids} for many other proofs of the same lemma).
We establish the positivity properties of $V_M$ on the {\em ample cone}, $\cK_M\subset L_M$ in Section~\ref{s:ample}, proving that $\cK_M$ has certain recursive properties with respect to contraction and restriction of matroids and that $\alpha$ and $\beta$ are in the closure of $\cK_M$.
In Section~\ref{s:symmetric}, we digress to prove some results on the eigenvalues of symmetric matrices, including the Perron--Frobenius theorem. Section~\ref{s:bootstrap} follows \cite{BL:oncones} to establish the Lorentzian bootstrap, giving conditions for the Hessian of derivatives of a homogeneous polynomial to have a unique positive eigenvalue. We also explain how this bootstrap relates to algebraic geometry. The Lorentzian bootstrap is applied to $V_M$ and from this, the Rota--Heron--Welsh conjecture is deduced in section~\ref{s:polc}. In section~\ref{s:generalizations}, we explain how the theory we outlined is a special case of two more general theories, Br\"{a}nd\'{e}n and Leake's theory of Lorentzian polynomials on cones \cite{BL:oncones} and Ross's Lorentzian fans \cite{Ross:Lorentzian}.

We have omitted a lot of the geometric motivation behind the theory and recommend \cite{Ardila:GoG} for a guide to the literature in that direction.

We'd like to thank Richard Ehrenborg, Christopher Eur, and Jonathan Leake for helpful comments.

\section{Chromatic Polynomials} \label{s:chromatic}

\subsection{Chromatic Polynomials}
Let $\Gamma$ be a  graph with finite vertex set $V(\Gamma)$ and finite edge set $E(\Gamma)$. We will take each edge with some orientation and let $e^+$ and $e^-$ denote its endpoints. A {\em proper coloring of $\Gamma$ with $q$ colors} for $q\in\Z_{\geq 1}$ is a function
\[c\colon V(\Gamma)\to \{1,2,\dots,q\}\]
such that for every edge $e$, $c(e^+)\neq c(e^-)$. That is, it is an assignment of one of $q$ colors to each vertex such that the endpoints of each edge must be different colors. If the graph has a loop edge, that is, an edge from a vertex to itself, then it has no colorings. The chromatic function $P_\Gamma\colon \Z_{\geq 1}\to \Z_{\geq 0}$ is defined by 
\[P_\Gamma(q)\coloneqq |\{\text{proper colorings of $\Gamma$ with $q$ colors}\}|\]
where ``$|\cdot|$'' denotes the cardinality of a set.

The function $P_G(q)$ is, in fact, a polynomial. That is, there is a polynomial $p_G(x)\in\Z[x]$ (indeed, with integer coefficients) such that for all $q\in\Z_{\geq 1}$, $P_G(q)=p_G(q)$. This observation is often proven through the deletion-contraction relation, but we will give a proof by decomposing the set of proper colorings into easier-to-count equivalence classes. For two proper colorings $c_1\colon V(\Gamma)\to \{1,2,\dots,q_1\}$ and $c_2\colon V(\Gamma)\to \{1,2,\dots,q_2\}$, we say $c_1$ and $c_2$ are equivalent (written $c_1\sim c_2$) if there exists a function $j\colon \{1,2,\dots,q_1\}\to \{1,2,\dots,q_2\}$ that is injective on the image of $c_1$ such that $j\circ c_1=c_2$. In other words, we can renumber the colors used in coloring $c_1$ to obtain $c_2$. A coloring $c\colon V(\Gamma)\to \{1,2,\dots,q\}$ is {\em surjective} if $c$ is surjective, i.e., it uses all $q$ colors. 

\begin{lemma} For any proper coloring $c_1$ of a graph $\Gamma$, there is a surjective proper coloring $c_2$ such that $c_1\sim c_2$.
\end{lemma}

\begin{proof} Set $S\coloneqq c_1(V(\Gamma))$. Write $q_2\coloneqq |S|$, fix a bijection $j\colon S\to \{1,\dots,q_2\}$, and set $c_2\coloneqq j\circ c_1$.
\end{proof}

In the above situation, we would say that {\em $c_1$ uses $|S|$ colors}.

\begin{lemma} There are finitely many surjective colorings of a graph $\Gamma$.
\end{lemma}

\begin{proof}
    Any surjective coloring $c\colon V(\Gamma)\to \{1,2,\dots,q\}$ must have $q\leq V(\Gamma)$. For such $q$, there are finitely many functions $c\colon V(\Gamma)\to\{1,2,\dots,q\}$.
\end{proof}

By combining the two lemmas, we see that 
there are only finitely many equivalence classes of colorings. We will count the number of proper colorings in an equivalence class $B$ with $q$ colors. Suppose that $B$ uses $q_B$ colors so that there is a surjection  $c\colon V(\Gamma)\to \{1,2,\dots,q_B\}$ in the class of $B$. Write $B_q$ for the set of proper colorings in $B$ with $q$ colors, so we can write $B$ as a disjoint union, $B=\sqcup_q B_q$. There is a bijection 
\begin{align*}
\left\{\text{injections }j\colon\{1,2,\dots,q_B\}\to \{1,2,\dots,q\}\right\}&\to B_q,\\ 
j&\mapsto j\circ c.
\end{align*}
Hence, there are 
\[(q)_{q_B}\coloneqq q(q-1)\dots (q-q_B+1)\]
elements of $B_q$. Observe that $(q)_{q_B}$ is a polynomial in $q$ with integer coefficients.
Now, the proper colorings of $\Gamma$ with $q$ colors can be expressed as a disjoint union $\sqcup_B B_q$ over equivalence classes, hence
\[P_\Gamma(q)=\sum_B |B_q|=\sum_B (q)_{q_B}.\]
Therefore, $P_\Gamma(q)$ is a polynomial with integer coefficients.

\begin{example}
Consider the graph consisting of a square and one of its diagonals. It has two equivalence classes of colorings (where one has three colors and the other has four):
\[
\xymatrix@R=7pt @C=7pt {
\bullet \ar@{-}[r]  \ar@{-}[d] \ar@{-}[dr] & \bullet \ar@{-}[d] \\
\bullet \ar@{-}[r] & \bullet
}\quad
\xymatrix@R=7pt @C=7pt {
1 \ar@{-}[r]  \ar@{-}[d] \ar@{-}[dr] & 2 \ar@{-}[d] \\
2 \ar@{-}[r] & 3
}\quad
\xymatrix@R=7pt @C=7pt {
1 \ar@{-}[r]  \ar@{-}[d] \ar@{-}[dr] & 2 \ar@{-}[d] \\
3 \ar@{-}[r] & 4
}
\]
Hence,
\[P_{\Gamma}(q)=(q)_3+(q)_4=q(q-1)(q-2)^2=q^4-5q^3+8q^2-4q.\]
\end{example}

We can give another approach to the polynomiality of $P_\Gamma(q)$ using a sophisticated version inclusion-exclusion. We begin by asking when a function $d\colon V(\Gamma)\to \{1,2,\dots,q\}$ fails to be a proper coloring. Let 
\[F_d\coloneqq \{e\in E(\Gamma)\mid d(e^+)= d(e^-)\},\]
be the set of edges that violate the proper coloring condition.
A function $d$ is a proper coloring exactly when $F_d$ is empty.
This set $F_d$ has the special property: it is a flat. A {\em flat} $F$ of a graph $\Gamma$ is a subset of edges $F\subset E(\Gamma)$ such that for any edge $e\in E(\Gamma)$, if there is a path from $e^+$ to $e^-$ through only edges of $F$, then $e\in F$. Indeed, if $e\in E(\Gamma)$ is such that there is a path in $F_d$ from $e^+$ to $e^-$, then we must have $d(e^+)=d(e^-)$ and so $e\in F_d$. 
Recall that a loop edge of a graph is an edge $e$ whose two endpoints are the same vertex $v$.  Observe that the set of all loop edges (denoted by $\hat{0}$) and the set $E(\Gamma)$ are both flats. The poset of flats of the above graph is shown in Figure~\ref{f:flats}. Here, the name ``flat'' is motivated by considerations in matroid theory.

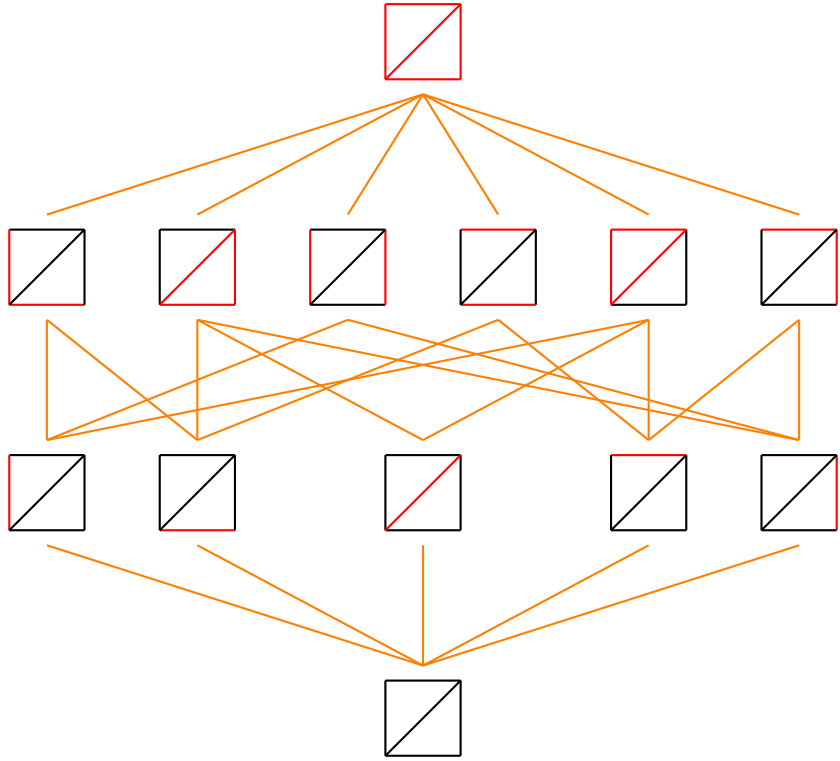
\begin{figure} \label{f:flats}             
  \centering
 \begin{tikzpicture}

\drawSquareFrom{5}{9}{red}{red}{red}{red}{red}

\drawSquareFrom{0}{6}{red}{black}{black}{red}{black}

\drawSquareFrom{2}{6}{red}{red}{black}{black}{red}

\drawSquareFrom{4}{6}{black}{red}{black}{red}{black}

\drawSquareFrom{6}{6}{red}{black}{red}{black}{black}

\drawSquareFrom{8}{6}{black}{black}{red}{red}{red}

\drawSquareFrom{10}{6}{black}{red}{red}{black}{black}

\drawSquareFrom{0}{3}{black}{black}{black}{red}{black}

\drawSquareFrom{2}{3}{red}{black}{black}{black}{black}

\drawSquareFrom{5}{3}{black}{black}{black}{black}{red}

\drawSquareFrom{8}{3}{black}{black}{red}{black}{black}

\drawSquareFrom{10}{3}{black}{red}{black}{black}{black}

\drawSquareFrom{5}{0}{black}{black}{black}{black}{black}

\orangeline{0}{0}

\orangeline{0}{4}

\orangeline{0}{8}

\orangeline{2}{0}

\orangeline{2}{2}

\orangeline{2}{6}

\orangeline{5}{2}

\orangeline{5}{8}

\orangeline{8}{6}

\orangeline{8}{8}

\orangeline{8}{10}

\orangeline{10}{2}

\orangeline{10}{4}

\orangeline{10}{10}

\loworangeline{0}
\loworangeline{2}
\loworangeline{5}
\loworangeline{8}
\loworangeline{10}

\highorangeline{0}
\highorangeline{2}
\highorangeline{4}
\highorangeline{6}
\highorangeline{8}
\highorangeline{10}


  \end{tikzpicture}

  \caption{The poset of flats (gives by edges marked in red) of a graph}
  \label{fig:test}
\end{figure}

We will make use of two properties of graphs, restriction and contraction:

\begin{definition}
  For a subset $S\subset E(\Gamma)$, we define the {\em restriction} $\Gamma^S$ to be the graph with vertices $V(\Gamma^S)\coloneqq V(\Gamma)$ and edges $S$.

  The {\em contraction} $\Gamma_S$ is defined as follows. The vertices $V(\Gamma_S)$ is the set of connected components of $\Gamma^S$.  The set of edges $E(\Gamma_S)$ is $E(\Gamma)\setminus S$, where an edge $e\in E(\Gamma_S)$ connects $v_1,v_2\in V(\Gamma_S)$ exactly when it connects the corresponding components of $\Gamma^S$.
\end{definition}

For a graph $\Gamma$, let $\kappa(\Gamma)$ be the number of connected components of $\Gamma$. Observe that 
\begin{alignat*}{2}
   |V(\Gamma^S)|=|V(\Gamma)|, &\quad  E(\Gamma^S)=S, \\
   |V(\Gamma_S)|=\kappa(\Gamma^S), &\quad E(\Gamma_S)=E(\Gamma)\setminus S,&\quad \kappa(\Gamma_S)=\kappa(\Gamma).    
\end{alignat*}

The next two lemmas are immediate.

\begin{lemma}
    A subset $S\subset E(\Gamma)$ is a flat if and only if $\Gamma_S$ has no loop edges.
\end{lemma}

\begin{lemma}
  Let $F$ be a flat of $\Gamma$. There is a bijection between the following:
  \begin{enumerate}
      \item functions $d\colon V(\Gamma)\to\{1,2,\dots,q\}$ with $F_d=F$, and
      \item proper colorings of the graph $\Gamma_F$.
  \end{enumerate}
\end{lemma}

\begin{lemma} \label{l:mobiusequality}
    The following equality holds
    \[q^{|V(\Gamma)|}=\sum_F P_{\Gamma_F}(q)\]
    where the sum is over all flats $F$ of $\Gamma$.
\end{lemma}

\begin{proof}
     Let $S_q$ denote the set of functions $d\colon V(\Gamma)\to \{1,2,\dots,q\}$. For a flat $F$, let 
     \[S_{q,F}\coloneqq \{d\in S_q\mid F_d=F\}.\]
     Since every coloring $d$ has a unique flat $F_d$, the set $S_q$ decomposes as a disjoint union over flats: $S_q=\bigsqcup_F S_{q,F}$.
     By taking the cardinality of this set equality, the conclusion follows.
\end{proof}

\begin{lemma} There is a polynomial $p_\Gamma(x)$ such that for $q\geq \Z_{\geq 1}$, $P_{\Gamma}(q)=p_{\Gamma}(q).$
\end{lemma}

\begin{proof}
  We induct on the number of edge of $\Gamma$. If $\Gamma$ has no edges, then $p_{\Gamma}=q^{\Gamma}$.
  
  For the inductive step, since proper colorings correspond to colorings $d$ with $F_d=\varnothing$, we have $P_{\Gamma}(q)=|S_{q,\varnothing}|$. Then 
  \[P_\Gamma(q)=q^{|V(\Gamma)|}-\sum_{F\neq \varnothing} P_{\Gamma_F}(q),\]
  which is  a polynomial.
\end{proof}

We have the following observation that the chromatic polynomial is characterized by the equality in Lemma~\ref{l:mobiusequality} together with a vanishing condition.
\begin{lemma} \label{l:characterizationchromatic}
    Let $f_\Gamma(q)\in\Z[q]$ be a collection of polynomials indexed by all finite graphs such that
    \begin{enumerate}
        \item if $\Gamma$ contains a loop edge, then $f_\Gamma(q)=0$, and
        \item for any graph $\Gamma$, $\sum_F f_{\Gamma_F}(q)=q^{|V(\Gamma)|}$ where the sum is over all flats of $\Gamma$,
    \end{enumerate}
    then for all $\Gamma$, $f_\Gamma(q)=P_\Gamma(q)$.
\end{lemma}

\begin{proof}
    The proof is a straightforward induction on the number of edges of $\Gamma$.
\end{proof}

We can write 
\[p_\Gamma(q)=\mu_0q^{|V(\Gamma)|}-\mu_1q^{|V(\Gamma)|}+\dots+(-1)^{|V(\Gamma)|}\mu_{|V(\Gamma)|}q^0\]
where the signs in front of the $\mu_i$'s are chosen to make each of them nonnegative as we shall see.

A consequence of our main result is:
\begin{theorem} \cite{Huh:chromatic}
    The chromatic polynomial $p_\Gamma(q)$ is log concave: for any $1\leq i\leq |V(\Gamma)|-1$,
    $\mu_{i-1}\mu_{i+1}\leq \mu_i^2.$
\end{theorem}

\section{Matroids} \label{s:matroids}

There is a more general object than a graph, called a matroid, which possesses a polynomial invariant similar to the chromatic polynomial. Indeed, the matroid pioneer, W.\ T.\ Tutte wrote \cite[p.~497]{Tutte:Selected2}
\begin{quote}
If a theorem about graphs can be stated in terms of edges and circuits only[,] it probably
exemplifies a more general theorem about matroids.
\end{quote}
We will slightly modify the definition of the chromatic polynomial in the next section so it uses only edges and flats (instead of circuits) and thus extend it to matroids.
Our description of matroids below is very brief and we recommend \cite{Oxley:matroid, Welsh:matroid} for more details.

\begin{definition} \label{d:flats} A matroid $M$ on a finite set $E$ is a collection of subsets $\cL(M)$ of $E$ that satisfy the following conditions
\begin{enumerate}
\item[(F1)] \label{i:everything} $E\in\cL(M)$,
\item[(F2)] \label{i:intersection} if $F_1,F_2\in\cL(M)$ then $F_1\cap F_2\in\cL(M)$, and
\item[(F3)] \label{i:covering} if $F\in\cL(M)$ and $\{F_1,F_2,\dots,F_k\}$ is the set of minimal members of $\cL(M)$ properly containing $F$ then the sets $F_1\setminus F,F_2\setminus F,\dots, F_k\setminus F$ partition $E\setminus F$.  \label{iflat:3}
\end{enumerate}
\end{definition}

Here, the last condition means that $E\setminus F$ is the disjoint union of the sets $F_1\setminus F$,$F_2\setminus F$, $\dots$, $F_k\setminus F$.
Usually, we let $E$ be the set $\{0,1,\dots,n\}$ for some nonnegative integer $n$. We say sets $I,J\subseteq E$ are incomparable if neither $I\subseteq J$ nor $J\subseteq I$.
When we have flats $F\subset G$ such that $G$ is a minimal flat properly containing $F$, we say $G$ {\em covers} $F$.
Because the set of flats is closed under intersection, there is a minimal flat $\hat{0}\in \cL(M)$ given by the intersection of all flats:
$\hat{0}=\cap_F F$. Elements of $\hat{0}$ are called {\em loops}. Write $\hat{1}\coloneqq E$ for the maximal flat. Define $\cP(M)\coloneqq\cL(M)\setminus\{\hat{0},\hat{1}\}$.

\begin{example}
    A natural class of matroids are the {\em uniform matroids} $U_{r,n}$ indexed by a nonnegative integer $n$ and an integer $r$ with $0\leq r\leq n$. Set $E\coloneqq \{1,2,\dots,n\}$, and let 
    \[\cL(U_{r,n})\coloneqq \{I\subseteq E \mid |I|\leq r-1\text{ or }I=E\}.\]
    It is straightforward to verify that this is a matroid.
\end{example}

\begin{example}
    Given matroids $M$ and $M'$ on sets $E,E'$ respectively, their direct sum is the matroid $M\oplus M'$ on $E\sqcup E'$ with
    \[\cL(M\oplus M')\coloneqq \{F\sqcup F'\mid F\in\cL(M), F'\in \cL(M')\}.\]
\end{example}

\begin{definition}
    Let $M$ and $M'$ be matroids on sets $E$ and $E'$, respectively. We say that $M$ and $M'$ are isomorphic if there is a bijection $f\colon E\to E'$ such that 
    \[\cL(M')=\{f(F)\mid F\in\cL(M)\}.\]
\end{definition}

\begin{remark} \label{r:hyperplanematroids}
Matroids can arise naturally from hyperplane arrangements, which are finite sets of hyperplanes in a vector space. Let $K$ be a field and let $\{H_0,\dots,H_n\}$ be a collection of hyperplanes in a finite-dimensional $K$-vector space $V$. We consider intersections of the $H_i$'s: for $I\subseteq E\coloneqq\{0,1,\dots,n\}$, write 
\[H_I\coloneqq\bigcap_{i\in I} H_i.\]
We define a matroid by giving the set of flats: 
\[\cL(M)\coloneqq \{F\subseteq E\mid \text{for all } j\in E\setminus F, H_{F\cup\{j\}}\neq H_F\}.\] 
One should view the set of flats as being the set of distinct intersections $\{H_I\}$. 
For a given intersection, we choose a unique such $I$ by taking the maximal index set yielding that intersection: for $L=H_I$, the attached flat is
\[F_L\coloneqq \{i\in E\mid L\subset H_i\}.\]
Observe that $I\subseteq F_L$.
It is a straightforward exercise that $\cL(M)$ is the set of flats for a matroid. Not all matroids (indeed, asymptotically none of them \cite{Nelson:matroids}) arise from a hyperplane arrangement over a field.
\end{remark}

\begin{lemma} Condition (F3) may be replaced with the following condition: 
\begin{enumerate}
    \item [(F3')] for any $F\in \cL(M)$ and $i\in E\setminus F$, there is a flat $G$ that covers $F$ and contains $i$.
\end{enumerate}
\end{lemma}

\begin{proof}
  Suppose $M$ is a matroid. We will show that $M$ satisfies (F3'). Let $F\in \cL(M)$ and $i\in E\setminus F$. Then, there exists a flat $F_j$ covering $F$ such that $i\in F_j$.

  Suppose $\cL(M)$ is a collection of sets satisfying (F1), (F2), (F3'). Let $F\in \cL(M)$. By (F3'), there are distinct flats $F_1,\dots,F_k$ covering $F$ such that $F_1\cup\dots\cup F_k=E$. We must show that for any distinct $j$ and $j'$, $(F_j\setminus F)\cap (F_{j'}\setminus F)=\varnothing$. Otherwise, we would have an inclusion of flats
  \[F\subsetneq F_j\cap F_{j'}\subseteq F_j.\]
  If $F_j\cap F_{j'}=F_j$ then $F\subsetneq F_j\subsetneq F_{j'}$ contradicting $F_{j'}$ covering $F$. Otherwise, $F_j\cap F_{j'}\subsetneq F_j$ contradicting $F_j$ covering $F$.
\end{proof}

\begin{lemma} If $F\subseteq G$ are flats in $\cL(M)$, then the poset interval,
\[[F,G]\coloneqq \{H\in\cL(M)\mid F\subseteq H\subseteq G\},\]
considered as the collection of sets
\[\{H\setminus F\mid H\in\cL(M)\text{ and } F\subseteq H\subseteq G\}\]
can be identified with the set of flats of a loopless matroid $M_F^G$ on the set $G\setminus F$.
\end{lemma}

The proof of the above is a straightforward verification. We will write $M^G$ for $M_{\hat{0}}^G$ and $M_F$ for $M_F^{\hat{1}}$.

\begin{lemma}
    The set of flats of a graph $\Gamma$ form a matroid, $M(\Gamma)$. 
\end{lemma}

\begin{proof}
Properties (F1) and (F2) are immediate. We will verify Property (F3'): for any flat $F$ and $e\in E\setminus F$, there is a unique minimal flat containing $F\cup \{e\}$. Indeed, let $G$ be the set of edges of $\Gamma$ whose endpoints are connected by a path within $F\cup\{e\}$. This is easily seen to be a flat of $\Gamma$. Moreover, any other flat containing $F\cup\{e\}$ must contain $G$.
\end{proof}

If the graph $\Gamma$ has a loop edge $e$, then every flat of $M(\Gamma)$ must contain $e$. Indeed, $F$ is a flat and if $e$ is adjacent to a vertex $v$, then there is the empty path from $e^+=v$ to $e^-=v$. Since the edges of this path are trivially contained in $F$, we must have $e\in F$.

For a flat $G$, one can verify that the lattice of flats of the matroid attached to the restriction, $M(\Gamma^G)$ is $M(\Gamma)^G$. Similarly, one can verify for a flat $F$, the lattice of flats attached to the contraction, $M(\Gamma_F)$ is $M(\Gamma)_F$. A consequence of these observations is the following lemma.

\begin{lemma}
    Let $\Gamma$ be a graph with flats $F$ and $G$ such that $F\subset G$. Then, 
    $M(\Gamma)_F^G=M((\Gamma^G)_F)$.
\end{lemma}

\begin{remark}
For a hyperplane arrangement $\{H_0,\dots,H_n\}$ in an $(r+1)$-dimensional vector space $V$ with matroid $M$ and flats $F$ and $G$, $M_F$ is the matroid attached to the hyperplane arrangement in the vector space $H_F$ given by $\{H_G\mid G\text{ covers } F\}$,
and $M^G$ is the matroid attached to the hyperplane arrangement in $V$ given by $\{H_i\mid i\in G\}$.
\end{remark}

Given a containment of flats $F\subset G$ in a matroid $M$, a {\em chain of length $k$} is a collection of flats
\[F=F_0\subsetneq F_1\subsetneq \dots\subsetneq F_k=G.\]
It is {\em maximal} if it cannot be refined by inserting more flats. 

\begin{lemma} \label{l:maximalchain}
    For any flats $F\subseteq G$, all maximal chains from $F$ to $G$ in $M$ have the same length.
    Consequently, there is a rank function $\rank\colon \cL(M)\to \Z_{\geq 0}$ such that 
    \begin{enumerate}
        \item $\rank(\hat{0})=0$ and
        \item if $F'$ covers $F$, then $\rank(F')=\rank(F)+1$.
    \end{enumerate}
\end{lemma}

\begin{proof}
  Suppose that there are flats $F$ and $G$ with maximal chains $\cF$ and $\cF'$ from $F$ to $G$ of different lengths. Make the choice of $F$ and $G$ such that the length of the shorter of the two maximal chains (say $\cF'$ after possibly interchanging $\cF$ and $\cF'$) is minimized, so $\length(\cF)>\length(\cF')$. By replacing $M$ with $M_F^G$, we may suppose that the chains are between $\hat{0}$ and $\hat{1}$. By replacing $E$ by $G\setminus F$ and fixing a bijection between $G\setminus F$ and $\{0,1,\dots,n\}$ for some positive integer $n$, We may suppose $E=\{0,1,\dots,n\}$. Write the chains as 
  \[\hat{0}\subsetneq F_1\subsetneq \dots \subsetneq F_{k-1}\subsetneq F_k=E,\ \hat{0}\subsetneq F'_1\subsetneq \dots\subsetneq F'_{k'-1}\subsetneq F'_{k'}=E.\]
  We claim $F_i\neq F'_j$ except for the following two cases: $(i,j)=(0,0)$; or $(i,j)=(k,k')$. Otherwise, if $F_i=F'_j$, since $k\neq k'$, we  must have $i\neq j$ or $k-i\neq k'-j$. In the first case, consider the maximal chains
  \[\hat{0}\subsetneq F_1\subsetneq \dots\subsetneq F_i,\ \hat{0}\subsetneq F'_1\subsetneq \dots\subsetneq F'_j=F_i\]
  which are of different length.
  In the second case, consider
  \[F_i\subsetneq \dots\subsetneq F_k,\ F'_j\subsetneq \dots\subsetneq F'_{k'}\]
  which are also of different length. In either case, the shorter chain is shorter than the chain $\cF'$ which contradicts our choice of $F$ and $G$.

  By relabeling the elements of $E$, may suppose $0\in F'_1\setminus \hat{0}$.
  We will prove that there is a maximal chain $\cF''$ of length at least $k$ from $\hat{0}$ to $\hat{1}$ such that $0\in F''_1$. 
  Set 
  \[i(\cF)\coloneqq \min\left(\{i\mid 0\in F_i\}\right).\]
  By possibly replacing $\cF$ by a chain of length at least $k$, we may suppose $\cF$ is chosen from maximal chains of length at least $k$ to minimize $i(\cF)$.
  If $i(\cF)=1$, then we are done.
  Otherwise, suppose $i\coloneqq i(\cF)\geq 2$. Let $G$ be the minimal flat containing $F_{i-2}\cup\{0\}$. Then since 
  \[F_{i-2}\cup\{0\}\subseteq G\cap F_i\subseteq G\]
  and $G\cap F_i$ is a flat containing $0$, we must have $G\cap F_i=G$. Therefore, $G\subseteq F_i$.
  We claim $G\neq F_i$. Suppose otherwise. Let $j\in F_{i-1}\setminus F_{i-2}$, so $F_{i-1}$ is the minimal flat of containing $F_{i-2}\cup\{j\}$. Then, $j\in (F_{i-1}\setminus F_{i-2})\cap (G\setminus F_{i-2})=F_{i-1}\setminus F_{i-2}$ contradicting flat axiom (F3).
  The only place where additional flats may be inserted in 
  \[\hat{0}\subsetneq F_1\subsetneq \dots\subsetneq F_{i-2}\subsetneq G\subsetneq F_i\subsetneq\dots\subsetneq F_k\]
  is between $G$ and $F_i$. We do so, if necessary, to produce a maximal chain $\cF''$,
  \[\hat{0}\subsetneq F_1\subsetneq \dots\subsetneq F_{i-2}\subsetneq G\subsetneq J_1\subsetneq \dots\subsetneq J_m\subsetneq F_i\subsetneq\dots\subsetneq F_k\]
  of length at least $k$.
  Since $i(\cF'')<i(\cF)$, we have a contradiction.

  Now, both $F_1$ and $F'_1$ are  minimal flats containing the element $0$. This gives the contradiction $F_1=F'_1$. 

  Therefore, for any flat $F$, the length of every maximal chain from $\hat{0}$ to $F$ is equal. We define $\rank(F)$ to be the length of that chain. It is straightforward to verify that $\rank$ has the desired properties. 
\end{proof}

\begin{definition}
    The {\em rank of a matroid} $M$ on a set $E$ is $\rank(M)\coloneqq \rank (E)$.
\end{definition}

Observe that $\rank(M\oplus M')=\rank(M)+\rank(M')$. Indeed, let 
\[\cF=\{\hat{0}\subsetneq F_1\subsetneq \dots \subsetneq F_{k-1}\subsetneq F_k=E\},\ \cF'=\{\hat{0}\subsetneq F'_1\subsetneq \dots\subsetneq F'_{k'-1}\subsetneq F'_{k'}=E'\}\]
be maximal chains of flats in $M$ and $M'$, respectively. Then
\[\cF=\{\hat{0}\subsetneq F_1\sqcup\hat{0}\subsetneq \dots \subsetneq F_k\sqcup \hat{0}\subsetneq F_k\sqcup F'_1\subsetneq \dots \subsetneq F_k\sqcup F'_{k'}=E\sqcup E'\}\]
is a maximal chain of flats in $M\oplus M'$.

\begin{lemma} \label{l:graphrank}
    For a graph $\Gamma$ with matroid $M(\Gamma)$, 
    \[\rank(G)=|V(\Gamma)|-\kappa(\Gamma^G)\]
    for $G\in\cL(M)$.
\end{lemma}

\begin{proof}
    By replacing $\Gamma$ with $\Gamma^G$, it suffices to show this for $G=\hat{1}$ since $|V(\Gamma)|=|V(\Gamma^G)|$.
    We prove this this by induction on $\rank(M(\Gamma))$. If $\rank(M(\Gamma))$ is $0$, then the only flat is $\hat{0}=\hat{1}=E(\Gamma)$. Hence, the only edges of $\Gamma$ are loop edges. Therefore, $\kappa(\Gamma)=|V(\Gamma)|$, and 
    \[|V(\Gamma)|-\kappa(\Gamma)=0=\rank(\hat{0}).\]

    Now, we consider the case where $\rank(M(\Gamma))=1$. Then, the only flats are $\hat{0}$ (corresponding to all the loop edges) and $\hat{1}$ (corresponding to all the edges of the graph). Let $e\in\hat{1}\setminus \hat{0}$ have endpoints $v$ and $w$. We claim that all other edges of $\hat{1}\setminus\hat{0}$ are between $v$ and $w$. Indeed, since $\hat{1}$ is the minimal flat containing $e$, all edges of $\hat{1}\setminus\hat{0}$ must have endpoints contained in $\{v,w\}$. Therefore, $\Gamma$ consists of a collection of vertices $v_0,\dots,v_{n+1}$ with edges consisting of loop edges and a collection of parallel edges between a pair of vertices. There are thus $n+1$ vertices and $n$ components, hence
  \[|V(\Gamma)|-\kappa(\Gamma)=n+1-n=1=\rank(\hat{1}).\]

    Now, let $\Gamma$ be a graph such that $\rank(M(\Gamma))\geq 2$. Let $F\in\cL(M(\Gamma))\setminus\{\hat{0},\hat{1}\}$. Then, by forming a maximal chain in $\Gamma$ by concatenating a maximal chain in $M(\Gamma)^F$ and one in $M(\Gamma)_F$ and applying Lemma~\ref{l:maximalchain}, we have
    \begin{align*}
      \rank(\hat{1})&=\rank(M(\Gamma)^F)+\rank(M(\Gamma)_F)\\
      &=\rank(M(\Gamma^F))+\rank(M(\Gamma_F))\\
      &=(|(V(\Gamma^F)|-\kappa(\Gamma^F))+(|V(\Gamma_F)|-\kappa(\Gamma_F))\\
      &=(|V(\Gamma)|-\kappa(\Gamma^F))+(\kappa(\Gamma^F)-\kappa(\Gamma))\\
      &=|V(\Gamma)|-\kappa(\Gamma). \qedhere
    \end{align*} 
\end{proof}

\begin{remark}
  For the matroid attached to a hyperplane arrangement as in Remark~\ref{r:hyperplanematroids}, the rank of a flat $F$ is given by 
  \[\rank(F)=\dim V-\dim H_F.\]
\end{remark}

\begin{definition}
  For a rank $r+1$ matroid $M$ and $i\in\{1,\dots,r\}$, the truncation $\Tr^i(M)$ is the matroid whose lattice of flats is
  \[\cL(\Tr^i(M))\coloneqq \{F\in\cL(M)\mid \rank(F)\leq r-i\}\cup \{E\}.\]
\end{definition}

It is a verification that $\Tr^i(M)$ is a matroid. Observe that $\Tr^{i+1}(M)=\Tr^1(\Tr^i(M))$. The rank of $\Tr^i(M)$ is $r+1-i$.

\begin{definition}
For a matroid $M$, the {\em flat graph} $\Gamma_M$ is the graph whose vertices are the elements of $\cL(M)\setminus\{\hat{0},\hat{1}\}$ with an edge between $F_1$ and $F_2$ if and only if $F_1\subsetneq F_2$ or $F_2\subsetneq F_1$.
\end{definition}

This graph models the containment of flats of $M$, and will later be used to establish the Lorentzian property of the volume polynomial of $M$.
 
\begin{lemma} \label{l:flatgraph}
    The flat graph of a matroid of rank at least $3$ is connected.
\end{lemma}

\begin{proof}
  We first observe that given a flat, $F_2\neq \hat{0}$, there is a rank $1$ flat $F_1$ such that $F_1\subseteq F_2$. Indeed, take $i\in F_2\setminus \hat{0}$ and let $F_1$ be the minimal flat containing $i$. Since $F_1\cap F_2$ is a flat containing $i$, then $F_1\cap F_2=F_1$ and so $F_1\subseteq F_2$.

  Now, it suffices to show that for any distinct rank $1$ flats $F_1$, $F'_1$, there is a proper flat containing both of them. Since $F_1\not\subseteq F'_1$, there is $i\in F_1\setminus F'_1$. Let $F_2$ be the minimal flat containing $F'_1\cup\{i\}$. Then, $F_2$ must be of rank $2$. Since $F_1\cap F_2$ is a flat containing $i$, $F_1\cap F_2\neq \hat{0}$. Hence, $F_1\cap F_2=F_1$, and so $F_1\subseteq F_2$.
\end{proof}

\section{The Characteristic polynomial} \label{s:characteristic}

\begin{definition} \label{d:charpoly}
    The {\em characteristic polynomial for matroids} is a collection of polynomials, $\chi_M(q)\in \Z[q]$, one for each matroid $M$ such that
      \begin{enumerate}
          \item if $M$ is isomorphic to $M'$, then $\chi_M(q)=\chi_{M'}(q)$,
          \item if $\hat{0}\neq \varnothing$, $\chi_M(q)=0$, and
          \item $\sum_{F\in\cL(M)} \chi_{M_F}(q)=q^{\rank(M)}$ \label{i:flatsum}.
      \end{enumerate}
\end{definition}

The sum condition expresses a form of inclusion-exclusion over flats. The characteristic polynomial exists and is unique by an induction argument analogous to Lemma~\ref{l:characterizationchromatic} on the size of the set $E$. 
We will write it for a matroid $M$ of rank $r+1$ as 
\[\chi_M(q)=\mu_0q^{r+1}-\mu_1q^r+\dots+(-1)^{r+1}\mu_{r+1}q^0.\]
Here, alternating signs are chosen to ensure that each $\mu_i$ is nonnegative.
Note that the above definition characterizes the evaluation of the characteristic polynomial at an element $q_0$ in a ring $R$, i.e., if there is a collection of elements $a_{M_F}\in R$ for each flat $F\in \cL(M)$ such that for all $F\in\cL(M)$,
\[\sum_{G\in\cL(M),G\supseteq F} a_{M_G}=q_0^{\rank(M)-\rank(F)},\]
then $a_{M_{\hat{0}}}=\chi_{M_{\hat{0}}}(q_0)$.
One sees explicitly that $\chi_{U_{0,0}}(q)=1$ and $\chi_{U_{1,1}}(q)=q-1.$ 

Observe that the characteristic polynomial of the matroid $M(\Gamma)$ attached to a graph $\Gamma$ is not  the same as the chromatic polynomial of $\Gamma$. Indeed, if $\Gamma$ has no edges, $\chi_\Gamma(q)=q^{|V(\Gamma)|}$ while $\chi_{M(\Gamma)}=1$. This is because the matroid $M(\Gamma)$ does not remember the number of vertices of $\Gamma$. However, there is a relationship:
\begin{lemma}
    Let $\Gamma$ be a graph, then 
    \[\chi_{\Gamma}(q)=q^{\kappa(\Gamma)}\chi_{M(\Gamma)}(q)\]
    where $\kappa(\Gamma)$ is the number of connected components of $\Gamma$.
\end{lemma}

\begin{proof}
    If $\Gamma$ has a loop edge, the conclusion is trivial.
    For a graph $\Gamma$, define $\chi'(q)=q^{\kappa(\Gamma)}\chi_{M(\Gamma)}(q).$ If $\Gamma$ has no edges, then $\chi_{M(\Gamma)}(q)=1$, $\kappa(\Gamma)=|V(\Gamma)|$, and
    $\chi'_{\Gamma}(q)=q^{|V(\Gamma)|}.$
    Since for any flat $F$, $\kappa(\Gamma_F)=\kappa(\Gamma)$, and so
    \[\sum_{F\in\cL(M(\Gamma))} \chi'_{\Gamma_F}(q)=\sum_{F\in\cL(M(\Gamma))}
    q^{\kappa(\Gamma)}\chi_{M(\Gamma)_F}(q)=q^{\kappa(\Gamma)} q^{\rank(M(\Gamma))}=q^{|V(\Gamma)|}.\]
    Because $\chi'_{\Gamma}$ satisfies the hypotheses of Lemma~\ref{l:characterizationchromatic}, $\chi'_{\Gamma}=\chi_{\Gamma}$ for all graphs $\Gamma$.
\end{proof}

\begin{remark}
  The characteristic polynomial has an immediate description in terms of hyperplane arrangements in vector spaces over finite fields \cite{Athanasiadis:characteristic}. Indeed, let $q$ be a prime power, i.e., $q=p^a$ for some positive integer $a$. 
  Let $H_0,\dots,H_n$ be hyperplanes in $\F_q^{r+1}$, so that $H_i$ is defined by $\{\ell_i=0\}$ for a nonzero linear expression
  \[\ell_i(\mathbf{x})=\sum_{j=1}^{r+1} a_{ij}x_j\]
  where $a_{ij}\in \F_q$, and $(x_1,\dots,x_{r+1})$ are the coordinates of $\mathbf{x}\in \F_q^{r+1}$.
  Write $W\coloneqq\cap_i H_i$, and $\kappa\coloneqq\dim W=r+1-\rank(E)$.
  Pick an algebraic closure $\overline{\F}$ of $\F_{q}$, so that for a positive integer $b$, we have a finite field $\F_{q^b}\subset \overline{\F}$. For such a finite field and $I\subseteq \{0,1,\dots,n\}$, we define the following set
  \[U_I(\F_{q^b})=\{\mathbf{x}\in \F_{q^b}^{r+1} \mid \ell_i(\mathbf{x})=0 \text{ if and only if }i\in I\}.\]
  In other words, we are counting points in $\F_{q^b}^{r+1}$ that are in $H_I$ but not in any $H_J$ for $J\supsetneq I$.
  Observe that $U_I(\F_{q^b})=\varnothing$ unless $I$ is a flat.

  We claim
  \[|U_{\hat{0}}(\F_{q^b})|=q^{b\kappa}\chi_M(q^b)\]
  for all positive integers $b$. Thus, when $\cap_i H_i=\{0\}$, the characteristic polynomial counts the points of $\F_{q^b}^{r+1}$ that are not in any hyperplane $H_i$.
  Indeed, for a fixed value of $b$, set $a_{M_F}\coloneqq q^{-b\kappa}|U_F(\F_{q^b})|$. For a flat $F$, set 
  \[W_F(\F_{q^b})\coloneqq \{\x\in \F_{q^b}^{r+1}\mid \ell_i(\mathbf{x})=0 \text{ if }i\in F\},\]
  i.e., $W_F(F_{q^b})$ is the set of $F_{q^b}$-points in the $(\rank(M)-\rank(F))$-dimensional linear subspace $H_F$.
  Then, for a positive integer $b$
  \[W_F(\F_{q^b})=\sqcup_{G\supseteq F} U_G(\F_{q^b})\]
  where the disjoint union is taken over the flats of the matroid attached to the hyperplane arrangement. Now,
  \[q^{b(r+1-\rank(F))}=|W_F(\F_{q^b})|=\sum_{G\supseteq F} |U_G(\F_{q^b})|.\]
  Therefore, multiplying by $q^{-b\kappa}$, we obtain
  \[q^{\rank(M)-\rank(F)}=\sum_{G\supseteq F} a_{M_G}.\]
  Hence, $a_{M_{\hat{0}}}=\chi_M(q^b).$

  Note that because the characteristic polynomial is a polynomial, it is determined by finitely many of its values, thus knowing $U_\varnothing(\F_{q^b})$ for sufficiently many values of $b$ determines it.
\end{remark}

\begin{remark}
  We may generalize the above remark by considering algebraic varieties over a field $k$. Let $\mathcal{V}$ be a collection of algebraic varieties defined over $K$ that contains hyperplane arrangement complements and is closed under disjoint union and Cartesian product. Let
  $\chi'\colon \mathcal{V}\to \Z[q]$
  be a function
  invariant under isomorphism such that 
  \[\chi'(\mathbb{A}_K^n)=q^n, \quad\chi'(X_1\sqcup X_2)=\chi'(X_1)+\chi'(X_2), \quad \chi'(X_1\times X_2)=\chi'(X_1)\chi'(X_2).\]
  Then, for a hyperplane arrangement $\{H_0,\dots,H_n\}$ in a finite dimensional vector space $V$ defined over $K$,
  \[\chi'(V\setminus (H_0\cup\dots\cup H_n))=q^{\dim (H_0\cap\dots\cap H_n)}\chi_M(q)\]
  where $M$ is the matroid of the hyperplane arrangement. Such a function can be realized by the Poincar\'{e} polynomial of compactly-supported cohomology 
 which must equal the characteristic polynomial \cite{OS:complements}. This point of view is discussed in \cite{Katz:MTFAG}.
\end{remark}

\begin{lemma}
    Let $M$ and $M'$ be matroids on sets $E$ and $E'$, respectively. Then, 
    \[\chi_{M\oplus M'}(q)=\chi_{M}(q)\chi_{M'}(q).\]
\end{lemma}

\begin{proof}
    We induct on $|E|+|E'|$. This is clearly true for $|E|=|E'|=0$. Using the inductive hypothesis, we see
    \begin{align*}
        \sum_{F\sqcup F' \in\cL(M\oplus M')} \chi_{(M\oplus M')_{F\sqcup F'}}(q)&=q^{\rank(M\oplus M')}\\
        &=q^{\rank(M)}q^{\rank(M')}\\
        &=\left(\sum_{F\in\cL(M)} \chi_{M_F}(q)\right)\left(\sum_{F'\in\cL(M')} \chi_{M_{F'}}(q)\right)\\
        &=\sum_{F\sqcup F' \in\cL(M\oplus M')} \chi_{M_F}(q)\chi_{M_{F'}}(q)\\
        &=\sum_{F\sqcup F' \in\cL(M\oplus M')\setminus \{\hat{0}\sqcup\hat{0}\}} \chi_{(M\oplus M')_{F\sqcup F'}}(q)+\chi_M(q)\chi_{M'}(q). 
    \end{align*}
    Thus, $\chi_{M\oplus M'}(q)=\chi_{M}(q)\chi_{M'}(q)$.
\end{proof}

Since  $U_{n,n}=U_{1,1}^{\oplus n}$, $\chi_{U_{n,n}}(q)=\chi_{U_{1,1}}(q)^n=(q-1)^n.$

We define the {\em M\"obius invariant of $M$} to be $\mu_M\coloneqq \chi_M(0)$. By substituting $q=0$ and an easy induction argument on $|E|$, we have the following:

\begin{lemma} \label{l:mobiusproperty}
    We have
    \begin{enumerate}
        \item if $M$ is isomorphic to $M'$, $\mu_{M}=\mu_{M'}.$
        \item if $\hat{0}\neq \varnothing$, $\mu_M=0$,
        \item if $E=\varnothing$, $\mu_M=1$, and 
        \item if $\rank(M)\geq 1$, then $\sum_{F\in\cL(M)}\mu_{M_F}=0.$
    \end{enumerate}
  These conditions characterize $\{\mu_M\}$ taken over all matroids $M$.
\end{lemma}

We will employ the incidence algebra \cite{Stanley:EC1}, an important tool in the study of finite posets, to make some observations about the M\"obius invariant.

\begin{definition}
Let $\cP$ be a finite poset, and define its incidence algebra $I_K(\cP)$ over a field $K$ \cite{Stanley:EC1} to be the algebra of formal sums
\[\sum_{x\leq y} a_{x,y}[x,y]\]
where $a_{x,y}\in K$ and $[x,y]$ is an interval in $\cP$. Multiplication is induced as a bilinear operation from 
\[[u,v]\cdot [x,y]=
\begin{cases}
  [u,y] &\text{if }v=x,\\
  0 &\text{else}.
\end{cases}\]
\end{definition}

It is straightforward to prove that $I_K(\cP)$ is an associative $K$-algebra with identity
$\delta=\sum_x [x,x]$.

\begin{remark}
    An example to have in mind is the poset of integers $\cP\coloneqq\{1,2,\dots,n\}$ under $\leq$. Here, $I_{\Q}(\cP)$ is isomorphic to the algebra of upper triangular matrices. For $i\leq j$, $[i,j]$ corresponds to the elementary matrix with a $1$ in the $i$th row and $j$th column and zeroes elsewhere.
\end{remark}

It is a straightforward exercise to show that an element $\alpha$ of $I_K(\cP)$ is invertible if and only if the coefficient of $[x,x]$ in $\alpha$ is nonzero for every $x\in \cP$. 

In particular, the {\em zeta function} $\zeta\in I_K(\cP)$,
\[\zeta\coloneqq\sum_{x\leq y} [x,y].\]
is invertible. For a matroid $M$ and $\cP=\cL(M)$, the inverse of $\zeta$ encodes the M\"obius invariants: set $\mu\in I(\cL(M))$ to be
\[\mu\coloneqq \sum_{F\subseteq G} \mu_{M_F^G} [F,G].\]
Then, Lemma~\ref{l:mobiusproperty} can be rephrased as $\zeta\mu=\delta,$
hence $\mu=\zeta^{-1}$. This implies $\mu\zeta=1$ which is equivalent to the following lemma.

\begin{lemma} \label{l:dualmobiusproperty}
  The M\"obius invariants of a matroid $M$ of rank at least $1$ obey
 \[\sum_{F\in\cL(M)}\mu_{M^F}=0.\]
\end{lemma}

Henceforth, we will write $r\coloneqq \rank(M)-1$.

\begin{lemma}
The characteristic polynomial of $M$ is given by the formula
 \[\chi_M(q)=\sum_{F\in\cL(M)}\mu_{M^F}q^{r+1-\rank(F)}.\] 
\end{lemma}

\begin{proof}
Write $\chi'_M(q)$ for the polynomial given by the above formula. We check that $\chi'_M(q)$ obeys the conditions of Definition~\ref{d:charpoly}. 
It is clearly an isomorphism invariant. If the minimal flat of $M$ is nonempty, the minimal flat of $M^F$ is also nonempty for each $F$, hence $\mu_{M^F}=0$, and so the above formula gives $\chi'_M=0$. Now,
\begin{align*}
    \sum_{F\in\cL(M)} \chi'_{M_F}(q)&=\sum_{F\in \cL(M)}\left(\sum_{G\in \cL(M_F)} (\mu_{M_F^G})q^{r+1-\rank(G)}\right)\\
    &=\sum_{G\in \cL(M)}\left(\sum_{F\leq G} \mu_{(M^G)_F}\right)q^{r+1-\rank(G)}=q^{r+1}. \qedhere
\end{align*}
\end{proof}

By Lemma~\ref{l:dualmobiusproperty}, if $\rank(M)\geq 1$, 
\[\chi_M(1)=\sum_{F\in\cL(M)}\mu_{M^F}=0.\]
In that case, we can define the {\em reduced characteristic polynomial} to be 
\[\overline{\chi}_M(q)=\frac{\chi_M(q)}{q-1}.\]
Write 
\[\overline{\chi}_M(q)=\mu^0q^r-\mu^1q^{r-1}+\dots+(-1)^r\mu^r.\]
We will later see, as a consequence of Lemma~\ref{l:mixedproduct} and Lemma~\ref{l:positivevolume}, that the $\mu^i$'s are nonnegative, justifying the alternating signs.

Let $[q^i]$ be the operator that extracts the coefficient of $q^i$ from a polynomial in $q$.

\begin{lemma} \label{l:reducedcharcoeff}
  We have the equality
  \[\mu^i=(-1)^{i+1}\sum_{\substack{F\in\cL(M)\\\rank(F)\geq i+1}}\mu_{M^F}.\]
\end{lemma}

\begin{proof}
    Observe that
    \[\overline{\chi}_M(q)=-\frac{\chi_M(q)}{1-q}= -\chi_M(q)(1+q+q^2+\dots),\]
  and 
  \begin{align*}
  \mu^i(M)&=(-1)^i[q^{r-i}](\overline{\chi}_M)(q)\\
  &=(-1)^{i+1}\left([q^0]\chi_M(q)+\dots+[q^{r-i}]\chi_M(q)\right)\\
  &=(-1)^{i+1}\sum_{\substack{F\in\cL(M)\\\rank(F)\geq i+1}}\mu_{M^F}. \qedhere
  \end{align*}
\end{proof}

\begin{lemma}
If $\rank(M)\geq 2$,
\[\mu^{r-1}(M)=\mu^{r-1}(\Tr^1(M))\]
for $r=\rank(M)-1$.
\end{lemma}

\begin{proof}
    Observe that
    \begin{align*}
        \mu^{r-1}(M)&=(-1)^r\left(\mu_M+\sum_{\substack{F\in\cL(M)\\\rank(F)=r}}\mu_{M^F}\right)\\
        &=(-1)^r\left(-\sum_{\substack{F\in\cL(M)\\\rank(F)\leq r}}\mu_{M^F}+\sum_{\substack{F\in\cL(M)\\\rank(F)=r}}\mu_{M^F}\right)\\
        &=(-1)^{r-1}\left(\sum_{\substack{F\in\cL(M)\\\rank(F)\leq r-1}}\mu_{M^F}\right)\\
        &=(-1)^{r}(\mu_{\Tr^1(M)})\\
        &=\mu^{r-1}(\Tr^1(M)). \qedhere
    \end{align*}
\end{proof}

By induction, we obtain:
\begin{lemma}\label{l:truncationmobius} We have 
$\mu^{r-i}(M)=\mu^{r-i}(\Tr^i(M)).$
\end{lemma}
Another easy consequence is
\[\overline{\chi}_{\Tr^1(M)}=\frac{\overline{\chi}_M(q)-\overline{\chi}_M(0)}{q}\]
as subtracting off the constant term and dividing by $q$ has the effect of shifting the coefficients.

We can now state the main result of this note, the Rota--Heron--Welsh conjecture (now the Adiprasito--Huh--Katz theorem):

\begin{theorem}\label{t:AHK}\cite{AHK} The characteristic polynomial is log-concave:
for $1\leq i\leq r$,
\[\mu_{i-1}\mu_{i+1}\leq \mu_i^2.\]
\end{theorem} 

\section{Why Log-concavity?} \label{s:why}

Why would one expect the characteristic polynomial to be log-concave and why does it matter? We briefly summarize some basic facts about log-concavity and recommend \cite{Branden:unimodality}, to which this section is indebted, for more details. 

\begin{definition}
    A finite sequence of real numbers $a_0,\dots,a_n$ is {\em unimodal} if there exists $j$ with $0\leq j\leq n$ such that
    \[a_0\leq \dots\leq a_{j-1}\leq a_j\geq a_{j+1}\geq\dots\geq a_n.\]
    The sequence is {\em log-concave} if for all $j$ with $1\leq j\leq n-1$,
    \[a_{j-1}a_{j+1}\leq a_j^2.\]
    A polynomial $p(x)\coloneqq a_0+a_1x+\dots+a_nx^n$ is {\em real-rooted} if all of its zeroes are real.
\end{definition}

For a fixed $n$, the binomial coefficients $\binom{n}{k}$ are log-concave, considered as a function of $k$. We say that the sequence $a_0,\dots,a_n$ is {\em ultra log-concave} if the sequence
\[a_0/\binom{n}{0},\dots,a_k/\binom{n}{k},\dots,a_n/\binom{n}{n}\] 
is log-concave. 

\begin{prop}
    For a sequence of nonnegative real numbers $a_0,\dots,a_n$, real-rootedness of $p(x)\coloneqq a_0+a_1x+\dots+a_nx^n$ implies the ultra log-concavity of $a_0,\dots,a_n$. Ultra log-concavity of a sequence implies log-concavity. Log-concavity implies unimodality.
\end{prop}

In deciding the existence of a coloring (say of a four-coloring of a planar graph), it is natural to consider the roots of a chromatic polynomial. Although the roots are rarely all real, after accounting for the alternating signs of the coefficients (by replacing $\chi_\Gamma(q)$ by $\pm\chi_{\Gamma}(-q)$), it is natural to ask if the coefficients exhibit weaker properties like log-concavity or unimodality. Unimodality of chromatic polynomials was observed in examples and then conjectured by Read. Work of Hoggar and then Rota, Heron, and Welsh extended the conjecture from unimodality to log-concavity, and then to the log-concavity of characteristic polynomials of matroids. See \cite{AHK} for references.

An important motivating question beyond the methods of \cite{AHK} was Mason's conjecture, the ultra log-concavity of the number of independent sets of a matroid. For a matroid $M$, a set $I\subset E$ is {\em independent} if the rank of the minimal flat containing it is $|I|$. One defines $f_k$ to be the number of independent sets of size $k$. The log-concavity of $\{f_k\}$, a conjecture of Mason and Welsh, was resolved using the log-concavity of the characteristic polynomial and an observation of Lenz. Ultra log-concavity was established independently by Brändén--Huh \cite{BH:Lorentzian} and Anari--Liu--Gharan--Vinzant \cite{ALGV} by Lorentzian techniques.

\section{Volume polynomials} \label{s:volume}

In this section, we define the volume polynomial of a matroid, a function 
\[V_M\colon\R^{\cP(M)}\to \R\] that encodes the characteristic polynomial and other invariants.

For a loopless matroid $M$ on a ground set $E$, write $\cP(M)\coloneqq \cL(M)\setminus\{\hat{0},\hat{1}\}$. Let $\R^{\cP(M)}$ be the vector space of functions $c\colon \cP(M)\to \R$. Write $\delta_F\colon \cP(M)\to \R$ for the function defined by
\[\delta_F(G)=\begin{cases}
    1&\text{if }G=F\\
    0&\text{else}.
\end{cases}\]
We introduce the following elements of $\R^{\cP(M)}$: if $i\in E$, let
\[\alpha_i\coloneqq \sum_{F\ni i} \delta_F,\quad \beta_i\coloneqq \sum_{F\not\ni i} \delta_F.\]
Note that $\alpha_i+\beta_i=\sum_F \delta_F$.
Let $W_M\subseteq \R^{\cP(M)}$ be the vector space 
\[W_M=\Span(\{\alpha_i-\alpha_j\mid i,j\in E\}),\]
and set $L_M\coloneqq \R^{\cP(M)}/W_M$. Under the projection $\R^{\cP(M)}\to L_M$, $\alpha_i$ and $\beta_i$ have images $\alpha$ and $\beta$, respectively, that are independent of $i$. 
We write an element of $\R^{\cP(M)}$ as
\[\t=\sum_F t_F\delta_F\]
for $t_F\in\R$. For $F\in \cP(M)$ and a differentiable function $V_M\colon \R^{\cP(M)}\to \R$, we  write
\[D_FV_M\coloneqq \frac{\partial\ }{\partial t_F} V_M,\ D_{\t}V_M\coloneqq \sum_F t_F D_F V_M.
\]

There are natural linear maps 
\[\pi^F\colon \R^{\cP(M)}\to L_{M^F}, \quad\pi_F\colon \R^{\cP(M)}\to L_{M_F}\]
defined  as follows:
\[\pi^F(\delta_G)=\begin{cases}
    \delta_G &\text{if }G\subsetneq F\\
    -\alpha & \text{if }G=F\\
    0 &\text{else}
\end{cases},
\quad\pi_F(\delta_G)=\begin{cases}
    \delta_G &\text{if }G\supsetneq F\\
    -\beta & \text{if }G=F\\
    0 &\text{else}
\end{cases}.
\]
By the following lemma, both of these maps will vanish on $\alpha_i-\alpha_j=\beta_j-\beta_i$ and, thus give well-defined maps
\[\pi^F\colon L_M\to L_{M^F},\quad\pi_F\colon L_M\to L_{M_F}.\]
It is immediate that $\pi^F\times\pi_F\colon L_M\to L_{M^F}\times L_{M_F}$ is surjective.
Below, we will consider the images of $\alpha_i$ and $\beta_i$ under $\pi_F$ and $\pi^F$. We will write $\alpha$ and $\beta$ for the corresponding element of $L_{M_F}$ and $L_{M^F}$ in the target space of $\pi_F$ and $\pi^F$ depending on context.

\begin{lemma} \label{l:alphabetaprojection}
    For $F\in\cP(M)$, we have the following for any $i\in E$,
    \begin{align*}
        \pi_F(\alpha_i)=\alpha,&\quad \pi^F(\alpha_i)=0,\\
        \pi_F(\beta_i)=0,&\quad \pi^F(\beta_i)=\beta        
    \end{align*}
\end{lemma}

\begin{proof}
  If $i\not\in F$, then we have
  \begin{align*}
  \pi_F(\alpha_i)&=\pi_F\left(\sum_{\substack{G\in \cP(M)\\G\ni i}} \delta_G\right)
  =\sum_{\substack{G\supseteq F\\G\ni i}} \pi_F(\delta_G)=\sum_{\substack{G\supsetneq F\\G\ni i}} \delta_G
  =\alpha,\\
  \pi^F(\alpha_i)&=\pi^F\left(\sum_{\substack{G\in \cP(M)\\G\ni i}} \delta_G\right)
  =\sum_{\substack{G\subseteq F\\G\ni i}} \pi^F(\delta_G)=0,\\
  \pi_F(\beta_i)& =\pi_F\left(\sum_{\substack{G\in \cP(M)\\G\not\ni i}} \delta_G\right)=
  \pi_F(\delta_F)+\sum_{\substack{G\supsetneq F\\G\not\ni i}} \pi_F(\delta_G)=-\beta+\beta=0,\\
   \pi^F(\beta_i)&=\pi^F\left(\sum_{\substack{G\in \cP(M)\\G\not\ni i}} \delta_G\right)=\pi^F(\delta_F)+\sum_{G\subsetneq F}\pi^F(\delta_G)
   =-\alpha+(\alpha+\beta)=\beta.
 \end{align*}
  
  If $i\in F$, 
  \begin{align*}
  \pi_F(\alpha_i)&=\pi_F\left(\sum_{\substack{G\in \cP(M)\\G\ni i}} \delta_G\right)
  =\pi_F(\delta_F) + \sum_{\substack{G\supsetneq F}} \pi_F(\delta_G)=-\beta+(\alpha+\beta)=\alpha,\\
  \pi^F(\alpha_i)&=\pi^F\left(\sum_{\substack{G\in \cP(M)\\G\ni i}} \delta_G\right)
  =\pi^F(\delta_F)+\sum_{\substack{G\subsetneq F\\G\ni i}} \pi^F(\delta_G)=-\alpha+\alpha=0,\\
  \pi_F(\beta_i)&=\pi_F\left(\sum_{\substack{G\in \cP(M)\\G\not\ni i}} \delta_G\right)
  =\sum_{\substack{G\supseteq F\\G\not\ni i}} \pi_F(\delta_G)=0,\\
  \pi^F(\beta_i)&=\pi^F\left(\sum_{\substack{G\in \cP(M)\\G\not\ni i}} \delta_G\right)
  =\sum_{\substack{G\subseteq F\\G\not\ni i}} \pi^F(\delta_G)
  =\sum_{\substack{G\subsetneq F\\G\not\ni i}} \delta_G=\beta.\qedhere\\
\end{align*}
\end{proof}

Hence, we can conclude that these linear maps are zero on $W_M$, and thus defined on $L_M.$ 

\begin{remark}
The values of $\pi_F$ and $\pi^F$ on $\delta_F$ is forced on us by their vanishing on $W_M$. Their definition was motivated by the self-intersection of a particular divisor in algebraic geometry.
\end{remark}

We can verify the following:
\begin{lemma} \label{l:mixedprojection}
  For flats $F\subsetneq G$, the following diagram commutes:
\[  \xymatrix{
  L_M\ar[r]^{\pi_{F}}\ar[d]_{\pi^{G}}&L_{M_{F}}\ar[d]^{\pi^{G}}\\
  L_{M^{G}}\ar[r]_{\pi_{F}}&L_{M_{F}^{G}},
  }\]
  i.e.,~$\pi^G(\pi_F(t))=\pi_F(\pi^G(t)).$
\end{lemma}

\begin{proof}
  We check the diagram on the image of $\delta_H$ for a flat $H$. Clearly if $H\not\in [F,G]$, the image of the two compositions is $0$. If $H\in (F,G)$ the image of the two compositions is $\delta_H$. It remains to check the identity on $\delta_F$ and $\delta_G$. Now,
  \[\pi^G(\pi_F(\delta_F))=\pi^G(-\beta)=-\beta=\pi_F(\delta_F)=\pi_F(\pi^G(\delta_F)).\]
  For $\delta_G$, we have
  \[\pi^G(\pi_F(\delta_G))=\pi^G(\delta_G)=-\alpha=\pi_F(-\alpha)=\pi_F(\pi^G(\delta_G)).\qedhere\]  
\end{proof}

Consequently, we will set $\pi_F^G\coloneqq\pi^G\circ \pi_F=\pi_F\circ \pi^G$ without ambiguity.

We will define the {\em volume polynomial} as  a function $V_M\colon L_M\to \R$. It is studied in greater depth from the point of view of matroid Chow rings in \cite{Eur:divisors}. 


\begin{definition}
    Let $f\in k[x_1,\dots,x_n]$ be a polynomial. The degree of a monomial $x_1^{d_1}x_2^{d_2}\cdots x_n^{d_n}$ is $d_1+\dots+d_n$. A polynomial is said to be homogeneous of degree $d$ if each of its monomials has degree $d$.
\end{definition} 

\begin{lemma}
    For each matroid $M$, there is a unique function $V_M\colon \R^{\cP(M)}\to\R$
    characterized by the following properties:
    \begin{enumerate}
        \item For a rank $1$ matroid $M$, $V_M=1$,
        \item if $\rank(M)>1$, then $V_M(0)=0$, and
        \item \label{i:diffeq} for any $F\in\cP(M)$, $D_F V_M(\t)=V_{M^F}(\pi^F(\t))V_{M_F}(\pi_F(\t))$.
    \end{enumerate}
    Moreover, $V_M$ is a homogeneous polynomial in $\{t_F\}$ of degree $\rank(M)-1$.
\end{lemma}

\begin{proof}
    We induct on the rank of $M$. For rank $1$, the case is clear. Suppose 
    $V_M$ has been constructed for all matroids of rank at most $r$. Let $M$ be a matroid of rank $r+1$. We first verify the equality of mixed partial derivatives: for any flats $F,G\in\cP(M)$, $D_F\left(D_GV_M\right)=D_G\left(D_FV_M\right)$. Then, by the usual existence theorem for differential equations, the function $V_M$ is well-defined.
    
    Let $F,G\in\cP(M)$. If $F$ and $G$ are incomparable, then
    \[D_G D_F V_M=
    D_G \left(V_{M^{F}}(\pi^{F}(t))\right)V_{M_{F}}(\pi_{F}(t))+
    V_{M^{F}}(\pi^{F}(t))D_G\left(V_{M_{F}}(\pi_{F}(t))\right)
    =0\]
    since $G$ is not a flat of either $M_F$ or $M^F$.
    Similarly $D_F D_G V_M=0.$ Now, suppose that $F$ and $G$ are comparable. We may suppose $F\subsetneq G$. Then,
    \begin{align*}
    D_G D_F V_M(\t)
    &=D_G\left(V_{M^F}(\pi^{F}(\t))V_{M_{F}}(\pi_{F}(\t))\right)\\
    &=D_G\left(V_{M^F}(\pi^{F}(\t))\right)V_{M_{F}}(\pi_{F}(\t))+
    V_{M^{F}}(\pi^{F}(\t))D_G\left(V_{M_{F}}(\pi_{F}(t))\right)\\
    &=0+V_{M^{F}}(\pi^{F}(t))V_{M_{F}^{G}}(\pi_F^G(\t))V_{M_G}(\pi_G(\t)).
    \end{align*}
    By an identical computation, this equals $D_F D_G V_M(\t)$.
    Since 
    \begin{align*}
      \deg(D_F V_M)&=\deg(V_{M^F})+\deg(V_{M_F})\\
      &=(\rank(M^F)-1)+(\rank(M_F)-1)\\
      &=\rank(M)-2,
    \end{align*}
    for all variables $t_F$,  $V_M$ has degree  equal to $\rank(M)-1$. Since $D_F V_M$ is homogeneous and $V_M(0)=0$, it follows that $V_M$ is homogeneous by induction on rank.
\end{proof}

To make \eqref{i:diffeq} hold for rank $1$ matroids, we use the convention that $V_M=0$ for a rank $0$ matroid $M$.

For $\rank(M)=2$, we have $V_M=\sum_{F\in \cP(M)} t_F$ as will be seen in Lemma~\ref{l:rank2volume}.

\begin{lemma} \label{l:derivativesofvolume}
    Let $\cF\coloneqq\{\hat{0}\subsetneq F_1\subsetneq F_2\subsetneq\dots\subsetneq F_{\ell}\subsetneq E\}$ be a chain of flats in $M$.
    Then,    
    \[D_{F_1}\cdots D_{F_{\ell}} V_M(\t)=V_{M^{F_1}}(\pi^{F_1}(\t))V_{M^{F_2}_{F_1}}(\pi^{F_2}_{F_1}(\t))\cdots V_{M^{F_\ell}_{F_{\ell-1}}}(\pi^{F_\ell}_{F_{\ell-1}}(\t))V_{M_{F_{\ell}}}(\pi_{F_{\ell}}(\t)).\]
    In particular, if $\cF$ is a maximal chain of flats with $\ell=r$, then 
    \[D_{F_1}\cdots D_{F_r} V_M=1.\]
\end{lemma}

\begin{proof}
    This is a straightforward induction on $\ell$ using the formula for $D_FV_M(\t)$.
\end{proof}

\begin{lemma}
  There is an equality of derivatives of $V_M$: for all $i,j\in E$,
  $D_{\alpha_i}V_M=D_{\alpha_j} V_M$.
  Consequently, the volume function $V_M$ factors 
        through the projection $\R^{\cP(M)}\to L_M$, inducing a map $V_M\colon L_M\to \R$.
\end{lemma}

\begin{proof}
   We induct on the rank of $M$. This is clear for $\rank(M)=1$. For $\rank(M)=2$, observe that for any $i\in E$,
   \[D_{\alpha_i} V_M(\t)=\sum_{G\ni i} \left(V_{M^G}(\pi^G(\t))\right)\left(V^{M_G}(\pi_G(\t))\right)=1\]
   since there is only one proper non-trivial flat containing $i$, and in this situation, both $M^G$ and $M_G$ are rank $1$ matroids.
   
   For $\rank(M)\geq 3$, we will show $D_{\alpha_i} V_M$ and $D_{\alpha_j} V_M$ satisfy the same differential equations with the same initial conditions. To check the initial conditions, observe
    \[\left(D_{\alpha_i} V_M\right)(0)=\sum_{G\ni i} \left(V_{M^G}(0)\right)\left(V^{M_G}(0)\right)=0\]
    and similarly for $\left(D_{\alpha_j} V_M\right)(0)$. Now, note
    \begin{align*}
      D_F D_{\alpha_i} V_M&=D_{\alpha_i}\left(\sum_F V_{M^F}(\pi^F(\t))V_{M_F}(\pi_F(\t))\right)\\
        & = \sum_F V_{M^F}(\pi^F(\t))D_{\alpha}V_{M_F}(\pi_F(\t))        
    \end{align*}
    where we used $\pi^F(\alpha_i)=0$ and $\pi_F(\alpha_i)=\alpha$. Because the last quantity is independent of $i$, $D_{\alpha_i} V_M$ and $D_{\alpha_j} V_M$ must be equal.   
    Consequently, $D_{\alpha_i-\alpha_j}V_M=0$. Since $\alpha_i-\alpha_j$ generate $W_M$,
    $V_M$ is constant along translates of $W_M$, and $V_M$ factors through $\R^{\cP(M)}\to L_M$.  
\end{proof}

\begin{lemma} \label{l:truncationvolume}
  Identify $\R^{\cP(\Tr^1(M))}$ with the subspace of $\R^{\cP(M)}$ consisting of functions $c\colon \cP(M)\to \R$ such that $c(F)=0$ for all $F\in \cL(M)$ with $\rank(F)=\rank(M)-1$.
  For any $i\in E$, 
  \[D_{\alpha_i}V_M\Big|_{\R^{\cP(\Tr^1(M))}}=V_{\Tr^1(M)}.\]
   Consequently, $D_{\alpha_i}^rV_M=1$ holds.
\end{lemma}

\begin{proof}
 
    We induct on $\rank(\Tr^1(M))=\rank(M)-1$.  If $\rank(\Tr^1(M))=1$, then $\rank(M)=2$, and
    \[D_{\alpha_i} V_M(\t)=D_F V_M(\t)=V_{M^F}(\pi^F(\t))V_{M_F} (\pi_F(\t))=1=V_{\Tr^1(M)} \]
    where $F$ is the unique proper flat of $M$ containing $i$.
    
    Now, consider the case where $\rank(\Tr^1(M))\geq 2$. We check that $D_{\alpha_i} V_M$ satisfies the same differential equations as $V_{\Tr^1(M)}$ with the same initial conditions. For  the initial condition, note 
    \[D_{\alpha_i} V_M(0)=\sum_{F\ni i}D_F V_M(0)=
    \sum_{F\ni i} V_{M^F}(\pi^F(0))V_{M_F}(\pi_F(0))=0.
    \]
    Now, let $F\in\cP(\Tr^1 M)$, and choose $i\in E\setminus F$
    \begin{align*}
        D_F D_{\alpha_i} V_M(\t)&=D_{\alpha_i}D_FV_M(\t)\\
        &=D_{\alpha_i}\left(V_{M^F}(\pi^F(t))V_{M_F}(\pi_F(t))\right)\\
        &=V_{M^F}(\pi^F(t))D_{\alpha_i}V_{M_F}(\pi_F(t))\\
        &=V_{M^F}(\pi^F(t))V_{\Tr^1(M_F)}(\pi_F(t))\\
        &=V_{\Tr^1(M)^F}(\pi^F(t))V_{\Tr^1(M)_F}(\pi_F(t)).
    \end{align*}
    This last quantity is equal to $D_F V_{\Tr^1(M)}(\t)$. Thus, $D_\alpha V_M$ satisfies the same differential equation as $V_{\Tr^1(M)}$. Hence, $D_\alpha V_M=V_{\Tr^1(M)}$.  

    By applying truncation $r$ times, we reach a rank $1$ matroid whose volume polynomial is $1$, hence  $D_{\alpha_i}^rV_M=1$.
\end{proof}

We will later use of the following {\em flat-avoidance lemma} which allows us to pick a representative of $\uu\in L_M$ that avoids a chain of flats:
\begin{lemma}
  Write $p\colon \R^{\cP(M)}\to \R^{\cP(M)}/W_M=L_M$ for the quotient map.
    Let 
    \[\cF=\{\hat{0}=F_0\subsetneq F_1\subsetneq\dots\subsetneq F_\ell\subsetneq E\}\]
    be a chain of flats and $\uu\in L_M$. Then, there exists $\t\in \R^{\cP(M)}$ such that $p(\t)=\uu$ and $t_{F_j}=0$ for all $j$.
\end{lemma}

\begin{proof}
  Pick $\t'\in\R^{\cP(M)}$ mapping to $\uu$. Pick $i\in E\setminus F_{\ell}$ and for $j=1,\dots,\ell$, pick $i_j\in F_j\setminus F_{j-1}$. Set
  \[\t=\t'-\sum_{j=1}^n (t'_{F_j}-t'_{F_{j-1}})(\alpha_{i_j}-\alpha_i).\]
  It is clear that $p(\t)=p(\t')$. Also, 
  \begin{align*}
      t_{F_k}&=t'_{F_k}-\sum_{j\colon i_j\in F_k} (t'_{F_j}-t'_{F_{j-1}})\\
      &=t'_{F_k}-\sum_{j=1}^k (t'_{F_j}-t'_{F_{j-1}})\\
      &=0.\qedhere
  \end{align*}
\end{proof}

The following lemma, originally proved in \cite{HuhKatz} is discussed in detail in \cite{Ardila:intersectiontheoryonmatroids}.

\begin{lemma} \label{l:mixedproduct} Let $r=\rank(M)-1$. Then $D_\alpha^{r-k} D_\beta^k V_M=\mu^k(M)$.
\end{lemma}

\begin{proof}
  The proof is by induction on $\rank(M)$.
  For a rank $1$ matroid, $V_M=1=\mu^0.$

  For the inductive step, we first prove this for the case $k=r$. We note
    \begin{align*}
     D_{\beta} V_M(\t)&=D_{\alpha+\beta} V_M(\t)-D_{\alpha} V_M(\t)\\
     &=\sum_{F\in \cP(M)} V_{M^F}(\pi^F(\t))V_{M_F}(\pi_F(\t))-V_{\Tr^1(M)}
     \end{align*}
  Because $\pi_F(\beta)=0$, the function $V_{M_F}(\pi_F(t))$ is constant along $\beta$, and therefore, $D_\beta\left(V_{M_F}(\pi_F(t))\right)=0$. Hence, we obtain by Lemma~\ref{l:reducedcharcoeff} and Lemma~\ref{l:dualmobiusproperty},
  \begin{align*}
    D_{\beta}^r V_M(\t)&= \sum_{F\in\cP(M)}\left(D_{\beta}^{r-1}V_{M^F}(\pi^F(t))\right)\left(V_{M_F}(\pi_F(\t))\right)-D_{\beta^{r-1}} V_{\Tr^1(M)}(\t)\\
     & = \sum_{\substack{F\in\cP(M)\\\rank(F)=r}}\left(D_{\beta}^{r-1}V_{M^F}(\pi^F(t))\right)\left(V_{M_F}(\pi_F(\t))\right)-D_{\beta^{r-1}} V_{\Tr^1(M)}(\t)\\
     & =\sum_{\substack{F\in\cP(M)\\\rank(F)=r}}\mu^{r-1}(M^F)-\mu^{r-1}(\Tr^1(M))\\
    &= (-1)^r\left(\sum_{\substack{F\in\cP(M)\\\rank(F)=r}}\mu_{M^F}-\mu_{\Tr^1(M)}\right)\\
    &= (-1)^r\left(\sum_{\substack{F\in\cP(M)\\\rank(F)=r}}\mu_{M^F} + \sum_{\substack{F\in\cP(M)\\\rank(F)\leq r-1}}\mu_{M^F}\right)\\
    &= (-1)^{r+1}\mu_M\\
    &=\mu^r(M).
  \end{align*}
  By Lemma~\ref{l:truncationmobius} and Lemma~\ref{l:truncationvolume},
\[     D_\alpha^{r-k}D_\beta^k V_M=D_\beta^kD_\alpha^{r-k} V_M=D_\beta^k V_{\Tr^{r-k}(M)}=\mu^k(\Tr^{r-k}(M))=\mu^k(M). \qedhere
\]
\end{proof}

\subsection{Low rank matroids} \label{ss:rank2and3}

We compute the volume polynomial of low rank examples.

\begin{lemma} \label{l:rank2volume}
 Let $M$ be a rank $2$ matroid. Then, the following holds
 \[V_M(\t)=\sum_{F\in \cP(M)} t_F.\]
\end{lemma}

\begin{proof}
    The only elements of $\cP(M)$ correspond to rank $1$ flats. For $p\colon \R^{\cP(M)}\to L_M$, $p(\delta_F)=\alpha$. Indeed, for $i\in F\setminus\hat{0}$, $\alpha_i=\delta_F$. Therefore, it suffices to compute $V_M(\alpha)$. Since $D_{\alpha}V_M=V_{\Tr^1 M}=1$ by Lemma~\ref{l:truncationvolume}, we must have $V_M(t_F\delta_F)=t_F$.
\end{proof}

Now, we compute the volume of a rank $3$ matroid $M$. Observe that $L_M$ is spanned by $\alpha$ together with $\delta_G$ for rank $2$ flats $G$. Indeed, for a rank $1$ flat $F$ with $i\in F$, we have in $L_M$, 
\[\delta_F=\alpha_i-\sum_{G\supsetneq F} \delta_G.\]

\begin{lemma} Let $M$ be a rank $3$ matroid. 
For $\t=t_1\alpha_1+\sum_G t_G\delta_G$ where the sum is over flats of rank $2$, the volume polynomial is
\[V_M(\t)=\frac{1}{2}\left(t_1^2-\sum_{\substack{G\\ \rank(G)=2}} t_G^2\right).\]
\end{lemma}

\begin{proof}
   We first compute the derivatives of $V_M$.
  We claim that for any distinct rank $2$ flats $G$ and $G'$,
  \begin{align*}
    D_\alpha^2V_M=1,&\quad D_GD_\alpha V_M=0,\\
    D_G^2 V_M=-1&,\quad D_{G'}D_{G}V_M=0
  \end{align*}

  The first equality is a consequence of Lemma~\ref{l:mixedproduct}. The second is
  \begin{align*}
      D_\alpha D_G V_M(\t)&=D_\alpha\left(V_{M^G}(\pi^G(\t))V_{M_G}(\pi_G(\t))\right))\\
      &=D_\alpha\left(V_{M^G}(\pi^G(\t))\right)\\
      &=0.
  \end{align*}
  since $\pi^G(\alpha)=0$ and, thus, $V_{M^G}(\pi^G(\t)$ is constant along $\alpha$. For the third equality,
  \begin{align*}
       D_G^2 V_M(\t)&=D_G\left(V_{M^G}(\pi^G(\t))V_{M_G}(\pi_G(\t))\right)\\
      &=D_G\left(V_{M^G}(\pi^G(\t))\right)\\
      &=-D_\alpha V_{M^G}(\pi^G(\t))\\
      &=-1
  \end{align*}
  where we used $\pi^G(\delta_G)=-\alpha$.
  Lastly, we have
    \begin{align*}
       D_{G'}D_G V_M(\t)&=D_{G'}\left(V_{M^G}(\pi^G(\t))V_{M_G}(\pi_G(\t))\right)\\
      &=D_{G'}\left(V_{M^G}(\pi^G(\t))\right)\\
      &=0
  \end{align*}
  which follows from $\pi^G(\delta_{G'})=0$. Since $D_{\alpha}^2V_M=1$, $D_G^2V_M=-1$ and all mixed terms vanish, we obtain the conclusion.
%
\end{proof}

\section{The ample cone} \label{s:ample}

We introduce a certain subset of $\cK_M\subset L_M$ called the {\em ample cone}. Here, a cone is a subset of a real vector space closed under addition and scalar multiplication by positive real numbers.
Now, $\cK_M$ will consist of elements on which $V_M$ will be shown to have strong positivity properties. It is the insight of \cite{BL:oncones} that one could study Lorentzian polynomials on cones rather than orthants. The ample cone is a natural candidate. Henceforth, for our convenience, we will suppose that $M$ is a loopless matroid.

A function $c$ on the subsets of $E$ is {\em strictly submodular} if
\[c(\varnothing)=0,\ c(E)=0,\]
and for any incomparable sets  $\varnothing \subsetneq I_1,I_2\subsetneq E$,
\[c(I_1)+c(I_2)>c(I_1\cap I_2)+c(I_1\cup I_2).\]
It is {\em submodular} if the non-strict inequality holds for all incomparable $\varnothing \subsetneq I_1,I_2\subsetneq E$. Observe that the sum of a submodular and a strictly submodular function is strictly submodular.
The {\em ample cone} $\cK_M\subset L_M$ is the set of all elements $y_c$ that can be expressed as
\[y_c\coloneqq \sum_{F\in\cL(M)} c(F)\delta_F\]
where $c$ is a strictly submodular function.

\begin{prop} \label{p:ampleconeproperties}
  The ample cone $\cK_M$ has the following properties:
  \begin{enumerate}
      \item $\cK_M$ is a nonempty open subset of $L_M$,
      \item $\cK_M$ is a cone: it is closed under addition and under scalar multiplication by positive real numbers.
  \end{enumerate}
\end{prop}

\begin{proof}
    Consider the function $c_+\colon\cP(M)\to \R$ given by 
    \[c_+(I)=|I|(|E|-|I|).\]
    We claim that it is strictly submodular. Indeed, let $I_1,I_2$ be incomparable sets with $\varnothing\subsetneq I_1,I_2\subsetneq E$. Write $k=|I_1\cap I_2|$, $n_1=|I_1|$, $n_2=|I_2|$, $n=|E|$, so $|I_1\cup I_2|=n_1+n_2-k$. By exchanging $I_1$ and $I_2$, we may suppose 
    \[k<n_1\leq n_2\leq n_1+n_2-k.\]
    Then,
    \[\begin{split}
       &c_+(I_1)+c_+(I_2)-c(I_1\cup I_2)-c(I_1\cap I_2)\\
       &=n_1(n-n_1)+n_2(n-n_2)-(n_1+n_2-k)(n-n_1-n_2+k)-k(n-k)\\
       &=2n_1n_2-2k(n_1+n_2-k)\hspace{4in}
    \end{split}\]
    The function $f(x)=x(n_1+n_2-x)$ is strictly increasing on $[0,(n_1+n_2)/2]$, so since $k<n_1$,  $f(k)<f(n_1)$ from which we conclude that the above quantity is positive.
    Hence, $c_+$ induces an element of $\cK_M$. Because the submodularity condition is defined by strict inequalities, $\cK_M\subset L_M$ is an open set.
    Since submodularity is closed under addition and preserved by positive scalar multiplication, $\cK_M$ is a cone.   
 \end{proof}

For a rank $1$ matroid, $\cK_M=L_M=\{0\}$. For a rank $2$ matroid $M$ on $E=\{0,1,\dots,n\}$, we claim that $\cK_M=\{t\alpha\mid t>0\}$. 
Let $F_1,\dots,F_k$ be the rank $1$ flats of $M$, and pick $i_j\in F_j$ for each $j$, so $\alpha_{i_j}=\delta_{F_j}$.
Since $y_{c_+}\in\cK_M$ is equal in $L_M$ to 
\[y_{c_+}=\sum_j c_+(F_j)\alpha_{i_j}=\left(\sum_j c_+(F_j)\right)\alpha,\]
which is a positive multiple of $\alpha$, every positive multiple of $\alpha$ is in $\cK_M$. 
An arbitrary element $y_c\in\R^{\cP(M)}$ is equal in $L_M$ to 
\[y_c=\sum_{j=1}^k c(F_j)\delta_{F_j}=\left(\sum_j c(F_j)\right)\alpha.\]
Now since $\cup_j F_j=E$, the $F_j$'s are pairwise disjoint, and $c(E)=0$, for any strictly submodular $c$, we see by downwards induction that
\[c(F_1)+c(F_2)+\dots+c(F_k)>c(F_1\cup F_2)+c(F_3)+\dots+c(F_k)>\dots>c(F_1\cup\dots\cup F_k)=0.\]
Consequently, $y_c\in\cK_M$ implies $y_c$ is equivalent to a positive multiple of $\alpha$.
In higher ranks $\alpha$ may not be in $\cK_M$, but it is in its closure.

The log-concavity of the characteristic polynomial will follow from the positivity properties of $V_M$ on $\cK_M$. To make use of them, we will need to know that $V_M$ has some positivity properties on the subspace of $L_M$ spanned by $\alpha$ and $\beta$. The following is critical:

\begin{lemma}
    The elements $\alpha_i$ and $\beta_i$ are in the topological closure of $\cK_M$ in $L_M$. 
\end{lemma}

\begin{proof}
  We first note that any submodular $c$, the element $y_c$ is in the closure of $\cK_M$. Indeed, for any $\varepsilon>0$, $c+\varepsilon c_+$ is strictly submodular. Since
  \[y_c=\lim_{\varepsilon\to 0} y_{c+\varepsilon c_+},\]
  $y_c$ is in the topological closure of $\cK_M$.

  Define $c_i\colon  2^E\to \R$ (where $2^E$ denotes the set of all subsets of $E$) by
  \[c_i(I)=\begin{cases}
      1 &\text{if }i\in I\\
      0 &\text{else}.
  \end{cases}\]
  Then $\alpha_i=y_{c_i}$. Observe that $c_i$ satisfies the submodularity condition with equality: for all $I_1,I_2$ with $\varnothing\subsetneq I_1,I_2\subsetneq E$,
  \[c_i(I_1)+c_i(I_2)=c_i(I_1\cap I_2)+c_i(I_1\cup I_2).\]

  The argument for $\beta_i$ is similar.
\end{proof}

Observe that $\alpha_i-\alpha_j$ maps to $0\in L_M$ under $p$ and so is in the closure of $\cK_M$.

The following is the combinatorial analogue of a fact from algebraic geometry: a very ample divisor is effective.

\begin{lemma} \label{l:ampleeffectivity}
  For any $\t\in \cK_M$, there is an element $\t'\in \R^{\cP(M)}$ mapping to $\t$ under $p\colon \R^{\cP(M)}\to L_M$ such that each component of $\t'$ is nonnegative and at least one is positive.
\end{lemma}

\begin{proof}
   Fix a bijection between $E$ and the set $\{0,1,\dots,n\}$. Write $[i]=\{1,\dots,i\}$ with  the convention that $[0]=\varnothing$.
   Let $\t\in \cK_M$ be induced by some strictly submodular function $c$. By induction on the size of a set $I$, we see
    \[\sum_{i\in I} c(\{i\})> c(I)\] and so
    \[c(\{0\})+\cdots +c(\{n\})>c(E)=0.\]
    By adding a multiple of $\alpha_1-\alpha_0$ to $\t$, we may suppose $c(\{0\})=0$.
    Hence 
    \[c'\coloneqq c-\sum_{j=2}^n (c([j])-c([j-1]))(c_j-c_1)\]
    differs from $c$ by an element that projects to $0$ in $L_M$, and hence it represents $\t$. Since both $c_j-c_1$ and $c_j-c_1$ are submodular, $c'$ is strictly submodular.
    Moreover, 
    \begin{align*}
    c'([i])&=c([i])-\sum_{j=2}^n (c([j])-c([j-1]))(c_j([i])-c_1([i]))\\
    &=c([i])+\sum_{j=2}^i (c([j])-c([j-1])) - \sum_{j=2}^n (c([j])-c([j-1]))\\
    &=c([i])-\sum_{j=i+1}^n (c([j])-c([j-1]))\\
    &=0.
    \end{align*}
    Given $I\subsetneq E$, we must show $c'(I)\geq 0$. We consider a number of cases. 
    First, note $c'(\{0\})=0$.
    Now, for any $I\subseteq [n]$, we induct on $\max(I)$, the maximal element of $I$. If $\max(I)=1$, $I=\{1\}$, and we are done. Otherwise, let $k=\max(I)$. If $I=[k]$, then we are done. Otherwise,
    \[c'(I)+c'([k-1])>c'(I\setminus \{k\})+c'([k])\geq 0.\]
     Also, if $I\not\subseteq [n]$ (i.e., $0\in I$),
    \[c'(I)+c'([n])>c'(I\setminus \{0\})+c'(E)\geq 0.\]
    Therefore, $c'$ is nonnegative. 
    Since some of the above inequalities are strict, at least one value of $c'$ is positive.
\end{proof}

The ample cone has good properties with respect to $\pi^F$ and $\pi_F$. The following is analogous to the algebraic geometric fact that the pullback of an ample divisor class is ample.

\begin{lemma} \label{l:restriction} For any $F\in\cL(M)$, $\pi^F(\cK_M)\subseteq \cK_{M^F}$ and $\pi_F(\cK_M)\subseteq \cK_{M_F}$.
\end{lemma}

\begin{proof}
  Let $\t\in \cK_M$ be represented by a strictly submodular $c$. Pick $i\in F, j\notin F$. Then $c'=c-c(F)(c_i-c_j)$ is a strictly submodular function representing $\t$ vanishing on $\delta_F$. Then 
  \[\pi^F\left(\sum_{G\in\cL(M)} c'(G)\delta_G\right)=\sum_{G\in\cL(M^F)} c'(G)\delta_G\]
  is induced by $c'$ considered as a function on subsets of $F$. As the restriction of a strictly submodular function (vanishing on $F$) to subsets of $F$, it is strictly submodular.
  Similarly, 
  \[\pi_F\left(\sum_{G\in\cL(M)} c'(G)\delta_G\right)=\sum_{G\in\cL(M_F)} c'(G)\delta_G\]
   is induced by the function on subsets of $E\setminus F$ given by $c_F(I)=c_F(I\cup F)$. Again, this is strictly submodular.
\end{proof}

The main input into the Lorentzian machinery will be an  observation about the derivatives of $V_M$ that will make use of the Eulerian identity:

\begin{lemma}
  Let $f\in k[x_1,\dots,x_n]$ be a homogeneous polynomial of degree $d$. 
  Write $\partial_i f$ for $\frac{\partial f}{\partial x_i}$. We have the {\em Eulerian identity}:
  \[\sum_{i=1}^n x_i\partial_if=df\]
\end{lemma} 

\begin{proof}
    Observe that for any monomial $x_1^{d_1}x_2^{d_2}\cdots x_n^{d_n}$,
\[\sum_{i=1}^n x_i\partial_i (x_1^{d_1}x_2^{d_2}\cdots x_n^{d_n})=
\sum_{i=1}^n d_ix_1^{d_1}x_2^{d_2}\cdots x_n^{d_n}=dx_1^{d_1}x_2^{d_2}\cdots x_n^{d_n}.\qedhere\]
\end{proof}

\begin{lemma} \label{l:positivevolume}
    For any matroid $M$ of rank $r+1$ and $\t\in \cK_M$, $V_M(\t)>0$.
    Also, for any $k\leq r$, $\t_1,\dots,\t_{k}\in\cK_M$,
    $D_{\t_1}\cdots D_{\t_k}V_M(\t)>0$.
\end{lemma}

\begin{proof}
  We induct on the rank of $M$. The conclusion is trivial if $\rank(M)=1$. We first prove $V_M(\t)>0$. We write  
  \[\t=\sum t_F\delta_F\]
  where each $t_F\geq 0$ for all $F$, and $t_F>0$ for at least one $F$ using Lemma~\ref{l:ampleeffectivity}.
  By the Eulerian identity,
  \[V_M(\t)=\frac{1}{r}\sum_F t_F D_F V_M(\t)\]
  and so
  \[V_M(\t)=\frac{1}{r}\sum_F t_F V_{M^F}(\pi^F(\t))V_{M_F}(\pi_F(\t)).\]
  Since $\pi^F(\t)\in \cK_{M^F}$ and $\pi_F(\t)\in \cK_{M_F}$, by induction, each evaluation of $V$ on the right side is positive, and so the sum is positive.

  For the second statement, we induct on $k$, noting that the case of $k=0$ has already been covered. Suppose $k\geq 1$ and represent
  \[\t_k=\sum a_F\delta_F\]
  where each $a_F\geq 0$ for all $F$, and for at least one $F$, $a_F>0$.
  Then 
  \[D_{\t_1}\cdots D_{\t_k}V_M(\t)=D_{\t_1}\cdots D_{\t_{k-1}}\sum_F a_F V_{M^F}(\pi^F(\t'))V_{M_F}(\pi_F(\t')).\]
  Now, for $A\subseteq [k-1]=\{1,\dots,k-1\}$, write $A=\{i_1,\dots,i_{\ell}\}$, and define
  \[D_{\t_A}=D_{\t_{i_1}}\cdots D_{\t_{i_\ell}}.\]
  By applying the Leibniz rule repeatedly, we see that the right side is 
  \[\sum_F \left(a_F \sum_{(A,B)\mid A\sqcup B=[k-1]} D_{\t_A}V_{M^F}(\pi^F(\t))D_{\t_B}V_{M_F}(\pi_F(\t)).\right)\]
  where the inner sum is over all disjoint pairs of subsets $(A,B)$ whose union is  $[k-1]$ such that $|A|\leq \rank(F)-1$ and $|B|\leq r-\rank(F)-1$.
  Since $\pi^F(\cK_M)\subseteq \cK_{M^F}$ and $\pi_F(\cK_M)\subseteq \cK_{M_F}$, each such term is positive by induction.
\end{proof}

By combining Lemma~\ref{l:positivevolume} with Lemma~\ref{l:mixedproduct} and the observation that $\alpha$ and $\beta$ are in the closure of $\cK_M$, we obtained the following:

\begin{corollary} The signed coefficients $\mu^i(M)$ of the reduced characteristic polynomial are nonnegative. 
\end{corollary}

\begin{remark}
  Note that, in general, $\cK_M\neq L_M$. One can see that by observing that for a rank $3$ matroid $M$ with rank $2$ flat $G$, we have $D_G^2V_M=-1$. This shows that $\delta_G\not\in \cK_M$.
\end{remark}

\section{Symmetric Matrices} \label{s:symmetric}

We now develop some tools from linear algebra to find conditions for a symmetric matrix to have a single positive eigenvalue. For the most part, we follow \cite{Serre:matrices} which we enthusiastically recommend.

Let $M_{n\times n}(\R)$ be the set of $(n\times n)$-matrices with real entries. A matrix $A$ is {\em symmetric} if $A^T=A$. Recall that a symmetric matrix has a complete set of real orthogonal eigenvectors. We write the eigenvalues (with multiplicity) in order as
\[\lambda_1(A)\leq \lambda_2(A)\leq\dots\leq \lambda_n(A).\]
The matrix $A$ induces a symmetric bilinear form,
\[\<\x,\y\>_A\coloneqq \x^TA\y.\]
If $\Lambda$ is an invertible $(n\times n)$ matrix, then
\[(\R^n,\<\cdot,\cdot\>_A)\to (\R^n,\<\cdot,\cdot\>_{\Lambda^T A \Lambda}),\quad \x\mapsto \Lambda^{-1}\x\]
is an isometery. In this case, we say $A$ and $\Lambda^T A \Lambda$ are {\em congruent}.
A matrix or vector is {\em nonnegative} (resp., positive) if all of its entries are. We will say that a matrix is {\em weakly nonnegative} (resp., weakly positive) if all of its off-diagonal are nonnegative (resp., positive).

\begin{definition}
The {\em incidence graph} $\IG(A)$ of a symmetric $(n\times n)$-matrix $A=[a_{ij}]$ is a graph whose vertices are labeled $1,2,\dots,n$ such that there is an edge between vertices $i$ and $j$ if and only if $a_{ij}\neq 0$, and there is a loop edge at vertex $i$ if and only if $a_{ii}\neq 0$. The matrix $A$ is said to be {\em irreducible} exactly when the incidence graph is connected.
\end{definition}

\begin{lemma}
A nonnegative matrix $A$ with positive entries along the diagonal is irreducible exactly when $A^k$ is positive for all sufficiently large $k$.
\end{lemma}

\begin{proof}
  View the incidence graph as a weighted graph where the edge from $i$ to $j$ is weighted by $a_{ij}$ and the loop edge at $i$ is weighted by $a_{ii}$.
  Because
  \[(A^k)_{ij}=\sum_{(i_1,\dots,i_{k-1})\in [n]^{k-1}} a_{ii_1}a_{i_1i_2}\cdots a_{i_{k-2}i_{k-1}}a_{i_{k-1}j}\]
  $(A^k)_{ij}$ is equal to the sum over all paths of length $k$ from $i$ to $j$ where each path is weighted with the product of the weights of its edges. Therefore, $(A^k)_{ij}$ is positive exactly when there is a path of length at most $k$ from $i$ to $j$. Because $\IG(A)$ has $n$ vertices, it has diameter at most $n-1$, so if there is a path from $i$ to $j$, there is a path of length at most $n-1$. Hence, $\IG(A)$ is connected if and only if $A^{n-1}$ is a positive matrix.
\end{proof}

We will need to understand the eigenvalues of symmetric matrices by studying how the inner product given by $A$ distorts lengths compared to the standard inner product.

\begin{definition} For a nonzero linear subspace $V\subset\R^n$, the {\em Rayleigh quotient} is 
\[R_A(V)\coloneqq \max_{\x\in V\setminus\{0\}} \frac{\<\x,\x\>_A}{\|\x\|^2}=\max_{\x\in V\setminus\{0\},\|\x\|=1} \<\x,\x\>_A.\]
\end{definition}

We will suppress ``$A$'' when the matrix is understood.
The following lemma makes use of the basic fact that for linear subspaces $V,W\subseteq \R^n$, 
\[\dim(V\cap W)\geq \max(0,\dim V+\dim W-n).\]

\begin{lemma} Let $A$ be a symmetric matrix with eigenvalues $\lambda_1\leq\lambda_2\leq\dots\leq\lambda_n$. Pick an orthonormal basis of eigenvectors $\v_1,\dots,\v_n$ with eigenvalues $\lambda_1,\dots,\lambda_n$, respectively.
For $V\subseteq\R^n$ with $\dim V=k$, $R(V)\geq\lambda_k$ with 
$R(\Span(\v_1,\dots,\v_k))=\lambda_k$. Hence, 
\[\lambda_k=\min_{V\mid \dim V=k} R(V).\]
\end{lemma}

\begin{proof}
Pick orthonormal eigenvectors $\v_1,\dots,\v_n$ as above. Set 
\[W=\Span(\v_k,\v_{k+1},\dots,\v_n).\] 
Let $V$ be a $k$-dimensional linear subspace of $\R^n$.
Since $\dim W=n+k-1$, $\dim V\cap W\geq 1$, so there is a nonzero $\vec{x}\in V\cap W\setminus\{0\}$. Since 
\[\x=a_k\v_k+\dots+a_n\v_n,\]
\[\frac{\<x,x\>_A}{\|x\|^2}=\frac{\lambda_k a_k^2+\lambda_{k+1}a_{k+1}^2+\dots +\lambda_n a_n^2}{a_k^2+\dots+a_n^2}\geq \frac{\lambda_k a_k^2+\lambda_{k} a_{k+1}^2+\dots +\lambda_k a_n^2}{a_k^2+\dots+a_n^2}=\lambda_k.\]
Hence, $R(V)\geq \lambda_k$.

For nonzero $x\in V\coloneqq\Span(\v_1,\dots,\v_k)$, 
\[x=a_1\v_1+\dots+a_k\v_k\] and
\[\frac{\<x,x\>_A}{\|x\|^2}=\frac{\lambda_1 a_1^2+\lambda_{2}a_{2}^2+\dots +\lambda_k a_k^2}{a_1^2+\dots+a_k^2}\leq \frac{\lambda_k(a_1^2+a_2^2+\dots+a_k^2)}{a_1^2+\dots+a_k^2}=\lambda_k.\]
Therefore, $R(V)\leq \lambda_k$, and hence $R(V)=\lambda_k$.
\end{proof}

As a special case, one  has the observation that for $\x\in\R^n$, $\<\x,\x\>_A\leq \lambda_n\|x\|^2$.

Recall that the signature of a symmetric matrix $A$ is the triple $(p,n,z)$ where $p$, $n$, and $z$ is the number of positive, negative, and zero eigenvalues. By Sylvester's law of inertia, two congruent matrices have the same signature.

\begin{lemma}
For $\v\in\R^n$, $\v\neq 0$, the matrix $B\coloneqq \v\v^T$ is a symmetric rank $1$ matrix with eigenvalues $0^{n-1},\|\v\|^2$ where $\v$ is an eigenvector with eigenvalue $\|\v\|^2$.
\end{lemma}

\begin{proof}
  The matrix $B$ is clearly symmetric. The nullspace $V=\v^\perp$  is $(n-1)$-dimensional. Since $B\v=\|\v\|^2\v$, $\v$ is an eigenvector with eigenvalue $\|\v\|^2$.
\end{proof}

Now, we study how subtracting $\v\v^T$ from a matrix changes its eigenvalues.

\begin{lemma} \label{l:eigenvalueshift} Let $A$ be a symmetric matrix with at most one positive eigenvalue. Let $\v\in\R^n$. Then $B\coloneqq A-\v\v^T$ has at most one positive eigenvalue. If $A$ is nonsingular, then $B$ has at most one nonnegative eigenvalue.
\end{lemma}

\begin{proof}
  Let $V$ be a $(n-1)$-dimensional subspace spanned by eigenvectors of $A$ corresponding to non-positive eigenvalues. Then $\<\cdot,\cdot\>_A$ is negative semi-definite on $V$. Since
  \[\<\x,\x\>_B=\<\x,\x\>_A-(\v^T\x)^2,\]
  $B$ is also negative semi-definite on $V$. Hence $\lambda_{n-1}\leq R(V)\leq 0$. If $A$ is non-singular, then $A$ and thus $B$ would be negative definite on $V$, and so $\lambda_{n-1}<0$. 
\end{proof}

The following is a special case of the Cauchy interlacing theorem:
\begin{lemma} \label{l:onepositive} Let $A$ be a symmetric matrix with at most one positive eigenvalue. Let $V\subset\R^n$ be a linear subspace with $\dim V\geq 2$. Then $\<\cdot,\cdot\>_A$ restricted to $V$ is not positive definite.
\end{lemma}

\begin{proof}
    This is trivial if $A$ has no positive eigenvalues.
    Let $\v_n$ be an eigenvector of $A$ corresponding to the positive eigenvalue. Then $V\cap \v_n^\perp$ is positive dimensional and contains a nonzero vector $\x$. Since $\x\in \v_n^\perp$, it lies in the span of the eigenvectors of $A$ with non-positive eigenvalues,
    \[\<\x,\x\>_A\leq \lambda_{n-1}\|\x\|^2\leq 0.\qedhere\]
\end{proof}

We include a proof of the Perron--Frobenius theorem.

\begin{theorem}
    Let $A$ be a weakly nonnegative irreducible symmetric matrix. Then $A$ has a unique largest eigenvalue. The corresponding eigenvector $\v$ can be chosen positive. All other positive eigenvectors are multiples of $\v$.
\end{theorem}

\begin{proof}
  By replacing $A$ by $A+\lambda I_n$ for $\lambda>0$ sufficiently large, we may assume that $A$ has positive entries along the diagonal. This has the effect of adding $\lambda$ to each eigenvalue without affecting the eigenvectors. Consequently, we may suppose that each eigenvalue of $A$ is positive.
  Let $r$ be a positive integer such that for any $i,j\in\{1,\dots,n\}$ there is a path in $I(G)$ from $i$ to $j$ of length at most $r$. Consequently, for any $i\in\{1,\dots,n\}$, $A^r\e_i$ is a positive vector. Therefore, for nonzero $\x\in\R^n_{\geq 0}$, $A^r\x$ is a positive vector. Hence, $A(\R^n_{\geq 0})\subseteq \R^n_{\geq 0}$ and $A^r(\R^n_{\geq 0}\setminus\{0\})\subseteq \R^n_{>0}$.
  Replacing $A$ by $A^r$ replaces $\lambda_i$ by $\lambda_i^r$ and does not affect the relative order the eigenvalues. We may therefore do so and suppose that $A$ is positive.

  To find the largest eigenvalue, we follow the dynamic approach of \cite{Birkhoff:linear}. We first show that the largest eigenvalue of $A$ has a positive eigenvector.
  Let $\lambda_1\leq \lambda_2\leq \dots\leq \lambda_n$ be the eigenvalues of $A$. Let $\v_1,\v_2,\dots,\v_n\in\R^n$ be a corresponding complete set of eigenvectors. Pick $\x\in\R^n_{>0}$ generically such that 
  \[\x=a_1\v_1+a_2\v_2+\dots+a_n\v_n\]
  where each $a_i$ is nonzero. Set $\uu_k=\frac{A^k\x}{\|A^k\x\|}$, so $\uu_k$ is a scalar multiple of $A^k\x=\lambda_1^k a_1\v_1+\lambda_2^k a_2\v_2+\dots+\lambda_n^k a_n\v_n$. Since $S^{n-1}\cap \R^n_{\geq 0}$ is compact, $\{\uu_k\}_k$ has a limit point $\uu$. Because $\lambda^ka_n$ must dominate all other coefficients of $A^k\x$, this limit point must be a scalar multiple of $\v_n$. Since $A\uu=\lambda_n\uu$, is a positive vector, $\uu$ must also be a positive vector.

  We follow the ideas of \cite[Section~8.3.3]{Serre:matrices} to show that the largest eigenvalue of $A$ is simple by showing the characteristic polynomial of $A$, $P_A(x)=\det(xI_n-A)$ satisfies $P'_A(\lambda_n)>0$. Now, since the determinant of a matrix is a multilinear function $\phi$ of its columns, 
  \[P_A(x)=\det(xI-A)=\det((xI-A)_1,\dots,(xI-A)_n)=\phi(x\e_1-A_1,\dots,x\e_n-A_n)\]
  where $A_i$ denotes the $i$th column of $A$,
  \[P'_A(x)=\phi(\e_1,xI-A_2,\dots,xI-A_n)+\dots+\phi(xI-A_1,\dots,xI-A_{n-1},\e_n).\]
  Write $B_i$ for matrix obtained from $A$ by deleting the $i$th row and $i$th column. This is a positive symmetric matrix. Then,
  \[P'_A(x)=P_{B_1}(x)+\dots+P_{B_n}(x).\]
  Let $j_i\colon \R^{n-1}\to \R^n$ be the inclusion into the $i$th coordinate hyperplane,
  \[j_i(x_1,\dots,x_{n-1})=(x_1,\dots,x_{i-1},0,x_i,\dots,x_{n-1}).\]
  We claim that the largest eigenvalue of $B_i$ is strictly less than $\lambda_n$. Let $\v\in\R_{>0}^{n-1}$ be an eigenvector of $B_i$ corresponding to the largest eigenvalue $\lambda'$ which is positive by induction. We will wiggle $\v$ in the $\e_i$-direction to increase its norm in $\<\cdot,\cdot\>_A$. Define a vector-valued function $\w\colon\R\to \R^n$, $\w(t)=j_i(\v)+t\e_i$ and a function
  \[Q(t)\coloneqq \frac{\w(t)^TA\w(t)}{\|\w(t)\|^2}=\frac{\lambda'\|\v\|^2+
  2t\e_i^TAj_i(\v)+t^2A_{ii}}{\|\v\|^2+t^2}.\]
  Now, 
  \[Q'(0)=\frac{2\e_i^TAj_i(\v)}{\|v\|^2}.\]
This quantity is positive since $A$ is a positive matrix and $j_i(\v)$ is a nonzero nonnegative vector. Since $Q(0)=\lambda'$, there is a vector $\w\coloneqq j_i(\v)+\varepsilon \e_i$ for small $\varepsilon>0$ such that $\frac{\<\w,\w\>_A}{\|\w\|}^2>\lambda'$. Hence, we have the inequalities
\[\lambda_n\geq \frac{\<\w,\w\>_A}{\|\w\|}^2>\lambda'.\]
  Since $P_{B_i}(x)\sim x^{n-1}$ for $x\gg 0$ and $P_{B_i}(x)$'s largest zero, $\lambda'$ is less than $\lambda_n$, $P_{B_i}(\lambda_n)>0$. Consequently, $P'_A(\lambda_n)>0$.

  Since the eigenspace $E_{\lambda_n}$ is $1$-dimensional, any eigenvector of $A$ is either parallel to $\v_n$ or orthogonal to it. In the latter case, the vector cannot be positive.
\end{proof}

\section{The Lorentzian bootstrap} \label{s:bootstrap}

The following section is a treatment of some material found in \cite[Section~2]{BL:oncones} and in \cite{BH:Lorentzian}.
For a homogeneous polynomial $f$, its gradient is the vector-valued function 
$\nabla f\colon\R^n\to \R^n$ given by $(\nabla f)_i(\x)\coloneqq\partial_i f(\x)$ and its Hessian $\cH(f)$ is the $(n\times n)$-matrix valued function $\cH f\colon \R^n\to \R^{n^2}$ whose components are $(\cH f(\x))_{ij}=\partial_i\partial_j f(\x)$. We will write $\nabla_{\x}f$ and $\cH_{\x}f$ for these functions evaluated at $\x\in\R^n$ when necessary. Note that if $\deg(f)=2$, then $\cH_{\x} f$ is independent of $\x\in\R^n$.

The goal of this section is to describe conditions for $\cH_{\x}f$ to have a unique positive eigenvalue. The arguments here are quite intricate and surprising.

\begin{lemma}
Let $f$ be homogeneous of degree $d$. Then we have the identities
\begin{enumerate}
    \item \label{i:ifunction} of functions: $(\nabla_{\x} f)^T\x=d f(\x)$ and
    \item \label{i:vv} of vector-valued functions: $(\cH_{\x}f)\x=(d-1)\nabla_{\x}f.$
\end{enumerate}
\end{lemma}

\begin{proof}
  Equation \eqref{i:ifunction} is the Eulerian identity.
  We prove the equality in \eqref{i:vv} component-by-component. Note
  \begin{align*}
      \e_i^T(\cH_{\x}f)\x&=\sum_j \frac{\partial^2 f}{\partial x_i\partial x_j}x_j\\
      &=\sum_j x_j\frac{\partial}{\partial x_j}\left(\frac{\partial f}{\partial x_i}\right)\\
      &=(d-1)\frac{\partial f}{\partial x_i}=(d-1)\e_i^T\nabla_{\x}f.\qedhere
  \end{align*}
 \end{proof}

Henceforth, we will let $f$ be a homogeneous polynomial of degree $d$.
Under certain hypotheses, we can leverage the above lemma to analyze the eigenvalues of a matrix $B_{\x}f$, a matrix congruent to $\cH_{\x}f$.  We will show the following:
\begin{enumerate}
    \item $B_{\x}f$ has $d-1$ as an eigenvalue with positive eigenvector,
    \item $B_{\x}f$ has no eigenvalues in the range $(0,d-1)$,
\end{enumerate}
If, in addition, $B_{\x}f$ is weakly nonnegative and irreducible, it has a unique positive eigenvector. Therefore, $B_{\x}f$ and thus $\cH_{\x}f$ have a unique positive eigenvalue.

Below, $\Lambda$ is a diagonal matrix with positive entries along the diagonal, and we take the positive square root.

\begin{lemma} Let $\x$ be a positive vector. Set $\v\coloneqq \nabla_{\x}f$ so $(\cH_{\x}f)\x=(d-1)\v$. Suppose $\x$ and $\v$ are both positive vectors.
Set
\[\Lambda\coloneqq \operatorname{diag}\left(\frac{x_1}{v_1},\dots,\frac{x_n}{v_n}\right),\quad B_{\x}f\coloneqq \Lambda^{1/2}(\cH_{\x}f)\Lambda^{1/2}.\]
Then, $\Lambda^{-1/2}\x$ is a positive eigenvector of $B_{\x}f$ with eigenvalue equal to $d-1$.
\end{lemma}

\begin{proof}
\[(B_{\x}f)(\Lambda^{-1/2}\x)=\Lambda^{1/2}(\cH_xf)\x=(d-1)\Lambda^{1/2}\v=(d-1)\Lambda^{-1/2}\Lambda\v=(d-1)\Lambda^{-1/2}\x.\qedhere \] 
\end{proof}

The following lemma will be critical in controlling the eigenvalues of $B_{\x}f$.
\begin{lemma} \cite[Proposition~2.33]{BH:Lorentzian} \label{l:negativesemi}
   Let $\x$ be a positive vector.
    Suppose that $\cH_\x f$ has a unique positive eigenvalue. Then the matrix
    \[C\coloneqq d\cdot f(\x)\cdot\cH_{\x}f-(d-1)\nabla_{\x} f\cdot (\nabla_{\x} f)^T\]
    is negative semi-definite.
\end{lemma}

\begin{proof}
We first suppose that $\cH_x f$ is nonsingular.
By Lemma~\ref{l:eigenvalueshift}, $C$ has at most one nonnegative eigenvalue. We will show that $C$ has a nontrivial kernel and thus a zero eigenvalue. Therefore, $C$ will not be able to have any positive eigenvalues. Indeed,
\begin{align*}
C\x&=d\cdot f(\x)\cdot\cH_{\x}(f)\x-(d-1)\nabla_{\x}f\cdot ((\nabla_{\x}f)^T\x)\\
&=d\cdot f(\x)\cdot (d-1)\nabla_{\x}f-(d-1)\nabla_{\x}f\cdot df(\x)\\
&=0. 
\end{align*}

In the general case, for $\varepsilon>0$, set
\[f_\varepsilon\coloneqq f-\varepsilon(x_1^d+\dots+x_n^d).\]
Since 
\[\cH_{\x}f_{\varepsilon}=\cH_{\x}f-d(d-1)\varepsilon\operatorname{diag}\left(x_1^{d-2},\dots,x_n^{d-2}\right),\]
we have for any linear subspace $V\subset\R^n$,
\[R_{\cH_{\x}f_{\varepsilon}}(V)\leq R_{\cH_{\x}f}(V)-\varepsilon d(d-1)m^{d-2}\]
for $m=\min(x_1,\dots,x_n)$. Therefore, 
\[\lambda_{n-1}(\cH_{x}f_{\varepsilon})\leq \lambda_{n-1}(\cH_x f)-d(d-1)\varepsilon m^{d-2}<0.\]
 Hence, for sufficiently small $\varepsilon$, $0$ is not an eigenvalue for $\cH_{\x}f_\varepsilon$. Write $C_\varepsilon$ for analogue of the matrix $C$ corresponding to $f_{\varepsilon}$. By the special case, $C_{\varepsilon}$ is negative semi-definite, and hence, so is $C=\lim_{\varepsilon\to 0} C_{\varepsilon}$ as the limit of negative semi-definite matrices.
\end{proof}

\begin{lemma}
  Suppose that $\x>0$, $\partial_i f(\x)>0$ and $\cH_{\x}\partial_i f$ has exactly one positive eigenvalue for all $i$. Then, $B_{\x}f$ has no eigenvalues in the range $(0,d-1)$. 
\end{lemma}

\begin{proof}
  We claim that it suffices to show that 
  \[(B_{\x}f)^2-(d-1)B_{\x}f\]
  is positive semi-definite.
  Indeed, if this is the case, let $\y$ be an eigenvector of $B_{\x}f$ with eigenvalue $\lambda$. Then
\[0\leq \y^T((B_{\x}f)^2-(d-1)B_{\x}f)\y=\y^T(\lambda^2-(d-1)\lambda)\y=(\lambda^2-(d-1)\lambda)\|\y\|^2.\]
   Hence $\lambda^2-(d-1)\lambda=\lambda(\lambda-(d-1))$ is nonnegative.

  By Lemma~\ref{l:negativesemi} applied to $\partial_i f$, the matrix
  \[(d-1)\cdot \partial_if(\x)\cdot\cH_{\x}(\partial_if)-(d-2)\nabla_{\x} (\partial_i f)\cdot (\nabla_{\x} (\partial_i f))^T\]
  is negative semi-definite. Multiplying this matrix by the positive real number $x_i/\partial_if(\x)$ and summing over $i$, we obtain the following negative semi-definite matrix:  
  \[\begin{split}
      &\sum_{i=1}^n \frac{x_i}{\partial_i f(\x)}\left((d-1)\cdot \partial_if(\x)\cdot\cH_{\x}(\partial_if)-(d-2)\nabla_{\x} (\partial_i f)\cdot (\nabla_{\x} (\partial_i f))^T\right)\\
      &= \sum_{i=1}^n \left((d-1)\cdot x_i\cdot\cH_{\x}(\partial_if)-(d-2)\nabla_{\x} (\partial_i f)\cdot  \frac{x_i}{\partial_i f(\x)} (\nabla_{\x} \partial_i f)^T\right)\\
      &=\left((d-1)(d-2) \cH_{\x}(f)-\sum_{i=1}^n (d-2)\nabla_{\x} (\partial_i f)\cdot  \frac{x_i}{\partial_i f(\x)} (\nabla_{\x} \partial_i f)^T\right)
  \end{split}\]
  where the last equality used the Eulerian identity applied to the entries of $\cH_xf$.
  Now, the $j$th component of $\nabla_{\x}(\partial_i f)$ is $\partial_j\partial_i f$. Consequently, we have
  \[
  \sum_{i=1}^n \left(\nabla_{\x} (\partial_i f)\cdot  \frac{x_i}{\partial_i f(\x)} (\nabla_{\x} \partial_i f)^T\right)=(\cH_{\x}f)\Lambda(\cH_{\x}f).
  \]
  Hence, $(\cH_{\x}f)\Lambda(\cH_{\x}f)-(d-1)\cH_xf$ is positive semi-definite. Multiplying on the left and right by $\Lambda^{1/2}$ yields the positive semi-definiteness of
  $(B_{\x}f)^2-(d-1)B_{\x}f$.
\end{proof}

\begin{prop} \label{p:bootstrap}
    Let $f$ be homogeneous of degree $d$. Let $\x$ be a positive vector such that
    \begin{enumerate}
        \item $\cH_{\x}f$ is a weakly nonnegative irreducible matrix,
        \item $\partial_i f(\x)>0$ for all $i$, and
        \item $\cH_{\x}(\partial_i f)$ has exactly one positive eigenvalue for all $i$, 
    \end{enumerate}
  then $\cH_{\x}f$ has a unique positive eigenvalue.
\end{prop}

\begin{proof}
  Consider $B_{\x}f$. Since $\Lambda$ is a diagonal matrix with positive entries, $B_{\x}f$ is weakly nonnegative and irreducible. Hence, it has a unique largest eigenvalue with positive eigenvector. Now, $B_{\x}f$ has a positive eigenvector, $\Lambda^{-1/2}\x$ with eigenvalue $(d-1)$ and has no eigenvalues in the range $(0,d-1)$. Hence, $B_{\x}f$ and thus, $\cH_xf$ has a unique positive eigenvalue.
\end{proof}

\begin{remark}
    The above proposition could be called the Lorentzian bootstrap. It abstracts some positivity properties from algebraic geometry (see e.g., \cite{Lazarsfeld:positivity1} for general background and \cite{Badescu} for an emphasis on the theory of surfaces). Let $X$ be a $d$-dimensional smooth projective algebraic variety. It has a numerical equivalence group $\Num(X)_{\Q}$ given by divisors (i.e.,~$\Q$-linear combinations of irreducible hypersurfaces up to the numerical equivalence relation). Within $\Num(X)_{\Q}$ is the ample cone $\cK_X$. There is a (normalized) volume function 
    $\Vol\colon (\Num(X)_{\Q}\otimes \Q)^{\otimes d}\to \Q$, a multilinear extension of  the intersection product
    \[\Vol(D_1\otimes\cdots \otimes D_d)\mapsto\frac{1}{d!}\left(D_1\cdot D_2\cdots D_d\right)
    \]
    for effective divisors $D_1,\dots,D_d\in\cK_M$ that intersect properly.

    If $X$ is an algebraic surface, the Hodge index theorem asserts that the matrix representing the intersection product
    \[D_1\otimes D_2\mapsto D_1\cdot D_2\]
    has a single positive eigenvalue. This can be phrased as the following: for an ample divisor $H$ and a numerically nontrivial divisor $D$, $D\cdot H=0$ implies $D^2<0$.

    The Hodge index theorem on surfaces induces the Khovanskii-Teissier inequality on $\Num(X)_{\Q}$ for $d$-dimensional $X$: for $D_1,D_2\in\cK_X$, 
    \[(D_1^d)(D_1^{d-2}D_2^2)\leq (D_1^{d-1}D_2)^2.\]
    We outline the proof. After replacing $D_1,D_2$ by positive integral multiples, one may suppose that they are very ample, hence realized by effective divisors. We may choose effective divisors $D_1^{(1)},\dots,D_1^{(d-1)}$, and $D_2^{(1)}$ such that $D_i^{(j)}$ is in the numerical equivalence class of $D_i$, and any subset of them meet transversely. Then, $Y\coloneqq D_1^{(1)}\cap\dots\cap D_1^{(d-2)}$ is an algebraic surface. Define the curves $Z_1,Z_2$ on $Y$ by $Z_1\coloneqq D_1^{(d-1)}\cap Y$ and $Z_2\coloneqq D_2^{(1)}\cap Y$.
    By the Hodge index theorem, the intersection pairing induced on $\Q Z_1\oplus \Q Z_2$, 
    \[(s_1Z_1+s_2Z_2)\cdot (t_1Z_1+t_2Z_2)=s_1t_1Z_1\cdot Z_1+(s_1t_2+s_2t_1)Z_1\cdot Z_2+s_2t_2Z_2\cdot Z_2\] 
    cannot be positive definite. The pairing is represented by the matrix
    \[\left[\begin{array}{cc}Z_1\cdot Z_1 & Z_1\cdot Z_2\\
    Z_2\cdot Z_1&Z_2\cdot Z_2
    \end{array}\right].\]
    Since $Z_1\cdot Z_1>0$, the determinant of the matrix must be non-positive from which one obtains
    \[(Z_1\cdot Z_1)(Z_2\cdot Z_2)-(Z_1\cdot Z_2)^2\leq 0\]
    and thus,
    \[(Z_1\cdot Z_1)(Z_2\cdot Z_2)\leq (Z_1\cdot Z_2)^2\]
    This can be rephrased as the following inequality of intersection numbers on $X$:
    \[(D_1^{d-2}\cdot D_1^2)(D_1^{d-2}\cdot D_2^2)\leq  (D_1^{d-2}\cdot D_1\cdot D_2)^2\]
    which is indeed the Khovanskii--Teissier inequality.

    One would like to prove this combinatorially. Unfortunately, the algebraic geometric technique of realizing divisor classes by transverse effective divisors does not easily translate. The insight of the Lorentzian bootstrap is that one can take a suitable basis $H_1,\dots,H_k$ of $\Num(X)$, realized as effective, irreducible divisors, and bootstrap the Hodge index theorem on surfaces arising as $(d-2)$-fold intersections of these divisors, subject to some local conditions. Here, the input is the Hodge index theorem on $H_{i_1}\cap\dots\cap H_{i_{d-2}}$, i.e.,the statement that the Hessian of $D_{H_{i_1}}\cdots D_{H_{i_{d-2}}}\left(\Vol\right)$ has a single positive eigenvalue. The output is that for arbitrary divisor classes $D_1,\dots,D_{d-2}\in \cK_X$, the Hessian of  $D_{D_1}\cdots D_{D_{d-2}}(\Vol)$ has a single positive eigenvalue, i.e., that the Hodge index theorem holds for the surface $D_1\cap \dots\cap D_{d-2}$ (at least for divisors pulled back from $X$). 
 \end{remark}

\section{Proof of Log-concavity} \label{s:polc}

In this section, we conclude the proof of the Adiprasito--Huh--Katz Theorem by deducing it from the Lorentzian signature of the Hessian of derivatives of $V_M$.

To prove irreducibility of Hessian matrices, we will need joins of graphs.
\begin{definition}
Given a graphs $\Gamma_1,\dots,\Gamma_k$, the {\em join} $\Gamma_1+\dots+\Gamma_k$ is the graph whose vertex set and edge sets are as follows:
\begin{align*}
    V(\Gamma_1+\dots+\Gamma_k)&=V(\Gamma_1)\sqcup\dots\sqcup V(\Gamma_k)\\
    E(\Gamma_1+\dots+\Gamma_k)&=E(\Gamma_1)\sqcup \dots\sqcup E(\Gamma_k)\sqcup\{vw\mid v\in V(\Gamma_i), w\in V(\Gamma_j) \text{ for }i\neq j\},
\end{align*}
that is, we take the disjoint union of $\Gamma_1,\dots,\Gamma_k$ and put an edge between each vertex of $\Gamma_i$ and each vertex of $\Gamma_j$ for $i\neq j$.
\end{definition}

We will deduce log-concavity from the following:

\begin{theorem} \label{t:uniquepositiveeigenvalue}
    For $\uu,\uu_1,\dots,\uu_k\in \cK_M$ with $k\leq \rank(M)-3$, $\cH_{\uu}(D_{\uu_1}\dots D_{\uu_k}V_M)$ has a unique positive eigenvalue. 
\end{theorem}

We explain the deduction before giving the proof. Suppose the above statement is true. Then, we have the following.
\begin{corollary} \label{c:logconcave}
  Let $\uu_1,\dots,\uu_{d-2}\in\cK_M$. Then, for any $\x,\y\in L_M$ with 
  \[D_{\x}^2 D_{\uu_1}\cdots D_{\uu_{d-2}}V_M>0,\] 
  we have the inequality
  \[(D_{\x}^2 D_{\uu_1}\cdots D_{\uu_{d-2}}V_M)(D_{\y}^2 D_{\uu_1}\cdots D_{\uu_{d-2}}V_M)
  \leq (D_{\x}D_{\y} D_{\uu_1}\cdots D_{\uu_{d-2}}V_M)^2.\]
\end{corollary}

\begin{proof}
  If $\x$ and $\y$ are parallel, the above is automatic with the inequality replaced by equality. Therefore, we may suppose that $\x$ and $\y$ span a $2$-dimensional subspace of $V_M$.
  Let $A\coloneqq \cH(D_{\uu_1}\dots D_{\uu_{d-2}}V_M)$ which has a single positive eigenvalue. Consider the inner product $\<\cdot,\cdot\>_A$ restricted to $\R\x\oplus \R\y$. It is represented by the matrix
  \[\begin{bmatrix}
      D_{\x}^2D_{\uu_1}\cdots D_{\uu_{d-2}}V_M& D_{\x}D_{\y}D_{\uu_1}\cdots D_{\uu_{d-2}}V_M\\
       D_{\x}D_{\y}D_{\uu_1}\cdots D_{\uu_{d-2}}V_M& D_{\y}^2D_{\uu_1}\cdots D_{\uu_{d-2}}V_M
  \end{bmatrix}.\]
  By Lemma~\ref{l:onepositive}, the above matrix cannot be positive definite. Since the upper-left corner is positive, it must have non-positive determinant, hence
  \[(D_{\x}^2 D_{\uu_1}\cdots D_{\uu_{d-2}}V_M)(D_{\y}^2 D_{\uu_1}\cdots D_{\uu_{d-2}}V_M)
  -(D_{\x}D_{\y} D_{\uu_1}\cdots D_{\uu_{d-2}}V_M)^2\leq 0. \qedhere\]
\end{proof} 

\begin{corollary}
    For all $1\leq i\leq r-1$,
    \[\mu^{i-1}\mu^{i+1}\leq (\mu^i)^2\]
    Moreover, the (non-reduced) characteristic polynomial is log concave.
\end{corollary}

\begin{proof}
    Let $\{\v_k\},\{\w_k\}$ be sequences of elements of $\cK_M$ converging to $\alpha$ and $\beta$ respectively. We apply the above Lemma with
    \[\uu_1=\dots=\uu_{i-1}=\w_k,\quad \uu_i=\dots=\uu_{r-2}=\v_k,\quad \x=\w_k,\quad \y=\v_k.\]
    Since $D_{\w_k}^{i+1}D_{\v_k}^{r-i-1}V_M>0$, we have
    \[(D_{\w_k}^{i-1}D_{\v_k}^{r-i+1}V_M)(D_{\w_k}^{i+1}D_{\v_k}^{r-i-1}V_M)\leq(D_{\w_k}^{i}D_{\v_k}^{r-i}V_M)^2.\]
    Taking the limit as $k$ goes to infinity gives $\mu^{i-1}\mu^{i+1}\leq (\mu^i)^2$.

    For the non-reduced characteristic polynomial, note
    \[\overline{\chi}_{M\oplus U_{1,1}}(q)=\frac{\chi_{M}(q)\chi_{U_{1,1}}(q)}{q-1}
    =\frac{\chi_{M}(q)(q-1)}{q-1}=\chi_M(q).\]
    Since $\chi_M(q)$ is equal to the reduced characteristic polynomial of $M\oplus U_{1,1}$, it is log-concave.
\end{proof}

We deduce Theorem~\ref{t:uniquepositiveeigenvalue} from a more general result.
For a chain of flats,
\[\cF\coloneqq\{\hat{0}=F_0\subsetneq F_1\subsetneq \dots \subsetneq F_{\ell}\subsetneq F_{\ell+1}=\hat{1}\},\] 
set 
    \[\R^{\cF}\coloneqq \R^{\cP(M_{F_0}^{F_1})}\times\dots\times \R^{\cP(M_{F_{\ell}}^{F_{\ell+1}})},\]
    and define a function $f_{\cF}\colon\R^{\cF} \to\R$
    by 
    \[f_{\cF}(\t_1,\dots,\t_{\ell})=V_{M^{F_1}}(\t_1)V_{M_{F_1}^{F_2}}(\t_2)\cdots V_{M_{F_{\ell-1}}^{F_\ell}}(\t_{l-1})V_{M_{F_{\ell}}}(\t_{\ell}). \]
Write $\pi_{\cF}\colon \R^{\cP(M)}\to \R^{\cF}$ for
\[\t\mapsto \left(\pi^{F_1}(\t),\pi_{F_1}^{F_2}(\t),\dots,\pi_{F_{\ell-1}}^{F_{\ell}}(\t),\pi_{F_{\ell}}(\t)\right),\]
so 
\[D_{F_1}\cdots D_{F_{\ell}}V_M(\t)=f_{\cF}(\pi_{\cF}(\t)).\]
Write $\cK_{\cF}\coloneqq \cK_{M_{F_0}^{F_1}}\times \dots\times\cK_{M_{F_{\ell}}^{F_{\ell+1}}}$.
\begin{prop}
    For $k\leq\deg(f_{\cF})-2=\rank(M)-\ell-3$ and 
    $\uu,\uu_1,\dots,\uu_k\in\cK_{\cF}$,
    $D_{\uu_1}\dots D_{\uu_k}f_{\cF}(\uu)$ is positive, and
    $\cH_{\uu}\left(D_{\uu_1}\dots D_{\uu_k}f_{\cF}\right)$ has a unique positive eigenvalue. 
\end{prop}

Because $f_{\hat{0}\subsetneq \hat{1}}=V_M$, this specializes to Theorem~\ref{t:uniquepositiveeigenvalue} when $\ell=0$.

\begin{proof}
    We induct on the degree of $f_{\cF}$. The base case is $\deg(f_{\cF})=2$ which  (after suppressing variables $\t_i$ in which $f_{\cF}$ is a constant) we have the following two cases: $f_\cF(\t)=f_M(\t)$ where $M$ is a rank $3$ matroid; or
    $f_{\cF}(\t_1,\t_2)=f_{M_1}(\t_1)f_{M_2}(\t_2)$ where $M_1$ and $M_2$ are rank $2$ matroids. In both of these cases, $k=0$. The first case is immediately covered by Subsection~\ref{ss:rank2and3}. In the second case, $f_{\cF}\colon\R^2\to \R$, given by $f_{\cF}(t_1,t_2)=t_1t_2$ has 
    \[\cH_{\uu}f_{\cF}=\left[\begin{array}{cc}0&1\\1&0\end{array}\right]\]
    which has exactly one positive eigenvalue.

    We first show $f_{\cF}(\uu)>0$. Write $r_i=\rank(F_i)-\rank(F_{i-1})-1$. By applying the Leibniz rule, $D_{\uu_1}\dots D_{\uu_k}f_{\cF}$
    is the sum of terms of the form
    \[\left(D_{\v_1}\cdots D_{\v_{k_1}}V_{M_{F_0}^{F_1}}\right)
    \left(D_{v_{k_1+1}}\cdots D_{v_{k_1+k_2}}V_{M_{F_1}^{F_2}}\right)
    \cdots 
    \left(D_{\v_{k_1+\dots+k_{\ell}+1}}\cdots D_{\v_k}V_{M_{F_{\ell}}^{F_{\ell+1}}}\right)\]
    where $(\v_1,\dots,\v_k)$ is a permutation of $(\uu_1,\dots,\uu_k)$ and
    each $k_i\leq r_i$. Each such term is positive by Lemma~\ref{l:positivevolume}.
    
    Let $S$ be the set of flats comparable to $\cF$, i.e.,
    \[S\coloneqq\cP(M_{F_0}^{F_1})\cup\dots\cup\cP(M_{F_{\ell}}^{F_{\ell+1}}),\]
    so that $\{\delta_F\}_{F\in S}$ forms a basis for $\R^{\cF}$.
    Let $F\in S$ with $F_{i-1}\subsetneq F\subsetneq F_{i}$ or some $i$. Then, write
    \[\cF'=\{\hat{0}=F_0\subsetneq\dots\subsetneq F_{i-1}\subsetneq F\subsetneq F_i\subsetneq \dots\subsetneq F_\ell\subsetneq F_{\ell+1}=\hat{1}\}\]
    and let $\pi_{\cF\to \cF'}\colon \R^{\cF}\to \R^{\cF'}$ be given by
    \[(\t_1,\dots,\t_{\ell})\mapsto \left(\t_1,\dots,\t_{i-1},\pi^F(\t_i),\pi_F(\t_i),\t_{i+1},\dots,\t_{\ell}\right).\]
    Observe that $D_Ff_{\cF}(\t)=f_{\cF'}(\pi_{\cF\to\cF'}(\t))$, and $\pi_{\cF\to\cF'}$ is surjective.
    We can generalize this construction by Lemma~\ref{l:mixedprojection}: if $\cF''$ is a refinement of $\cF$, there is a linear map $\pi_{\cF\to \cF'}\colon \R^{\cF}\to\R^{\cF''}$.

    Note that for any $F\in S$,
    \begin{align*}
    D_F\left(D_{\uu_1}\cdots D_{\uu_k}f_{\cF}\right)(\uu)
    &=D_{\uu_1}\cdots D_{\uu_k}D_Ff_{\cF}(\uu)\\
    &=D_{\uu_1}\cdots D_{\uu_k}\left(f_{\cF'}(\pi_{\cF\to\cF'}(\uu))\right)\\
    &=\left(D_{\pi_{\cF\to\cF'}(\uu_1)}\cdots D_{\pi_{\cF\to\cF'}(\uu_k)}f_{\cF'}\right)(\pi_{\cF\to\cF'}(\uu))>0
    \end{align*}
    where the last equality is the chain rule and the inequality is a consequence of $\pi_{\cF\to\cF'}(\cK_{\cF})\subseteq \cK_{\cF'}$ and the inductive hypothesis.

    Let $\Gamma_i$ be the flat graph of $M_{\Gamma_{i-1}}^{\Gamma_i}$.
    We consider $\cH_{\uu}(D_{\uu_1}\dots D_{\uu_k}f_{\cF})$ to be the Hessian with respect to the basis $\{\delta_F\}_{F\in S}$.
    We claim that the incidence graph $\Gamma$ of $\cH_{\uu}(D_{\uu_1}\dots D_{\uu_k}f_{\cF})$ is the join of $\Gamma_1,\dots,\Gamma_{\ell}$. We first show that there is an edge between $F,G\in S$ in $\Gamma$ if and only $F$ and $G$ are comparable.
    Indeed, choose distinct $F,G\in S$. If $F$ and $G$ are comparable, then let $\cF''$ be obtained from $\cF$ by adjoining $F$ and $G$. Then,
    \[D_F D_GD_{\uu_1}\cdots D_{\uu_k}f_{\cF}(\uu)=
    \left(D_{\pi_{\cF\to \cF''}(\uu_1)}\cdots D_{\pi_{\cF\to \cF''}(\uu_k)}f_{\cF''}\right)(\pi_{\cF\to \cF''}(\uu))>0\]
    If $F$ and $G$ are not comparable,
    \[D_F D_GD_{\uu_1}\cdots D_{\uu_k}f_{\cF}(\uu)=0.\]
    Now, to show $\Gamma=\Gamma_1+\dots+\Gamma_{\ell}$, we consider two cases:
    there exists $i$ such that $F_{i-1}\subsetneq F,G\subsetneq F_i$ or there does not. In the first case, $F$ and $G$ are comparable if and only if they are joined by an edge of $\Gamma_i$.
    Otherwise, after interchanging $F$ and $G$, we may suppose there exists $i<j$ such that 
    \[F_{i-1}\subsetneq F\subsetneq F_i,\quad F_{j-1}\subsetneq G\subsetneq F_j,\]
    so $F$ and $G$ corresponds to vertices of $\Gamma_i$ and $\Gamma_j$, respectively. 
    Now, $F$ and $G$ are comparable, and hence adjacent in the incidence graph by the previous case. 
    
    We claim that $\Gamma$ is connected. Indeed, $\Gamma_i$ is connected as long as $M_{\Gamma_{i-1}}^{\Gamma_i}$ is of rank at least $3$ by Lemma~\ref{l:flatgraph}, and the join of at least two nonempty graphs is always connected.  Thus, the only way in which $\Gamma$ could be disconnected is if $f_{\cF}$ is of degree $1$, which was already covered. Hence, $\cH_{\uu}(D_{\uu_1}\dots D_{\uu_k}f_{\cF})$ is an irreducible matrix. By the above computation of $D_F D_GD_{\uu_1}\cdots D_{\uu_k}f_{\cF}(\uu)$, we see that  $\cH_{\uu}(D_{\uu_1}\dots D_{\uu_k}f_{\cF})$ is weakly nonnegative.
    
    We claim for any $F\in S$, $\cH_{\uu}(D_F D_{\uu_1}\dots D_{\uu_k}f_{\cF})$ has exactly one positive eigenvalue. Indeed, let $\cF'$ be obtained from $\cF$ by adjoining $F$. Then
    \[D_FD_{\uu_1}\cdots D_{\uu_k}f_{\cF}=D_{\uu_1}\cdots D_{\uu_k}(f_{\cF'}\circ \pi_{\cF\to\cF'})\]    
    Hence for $\v,\w$, writing $\pi$ for $\pi_{\cF\to \cF'}$, we have
    \[\v^T\cH_{\uu}(D_FD_{\uu_1}\cdots D_{\uu_k}g)\w=(\pi(\v))^T\cH_{\pi(u)}(D_{\pi(\uu_1)}\cdots D_{\pi(\uu_k)}f_{\cF'})(\pi(\w)).\]
    Now, the Hessian on the right has a single positive eigenvalue by induction. We know that $\pi$ is surjective. Choose a linear map $j\colon \R^{\cF'}\to \R^{\cF}$ with $\pi\circ j=\operatorname{Id}_{\R^{\cF'}}$. Then 
    \[\R^{\cF}\cong \ker(\pi)\oplus j(\R^{\cF'}).\]
    The inner product induced by $\cH_{\uu}(D_FD_{\uu_1}\cdots D_{\uu_k}f_{\cF}))$ is $0$ on $\ker(\pi)$ and is isometric to $\cH_{\pi(u)}(D_{\pi(\uu_1)}\cdots D_{\pi(\uu_k)}f_{\cF'})$ on $j(\R^{\cF'})$. Since the latter has a unique positive eigenvalue, so does $\cH_{\uu}(D_FD_{\uu_1}\cdots D_{\uu_k}g)$.
    
    Since we have verified the conditions for the Lorentzian bootstrap, by Proposition~\ref{p:bootstrap}, the matrix $\cH_{\uu}(D_{\uu_1}\dots D_{\uu_k}f_{\cF})$ has a unique positive eigenvalue.
\end{proof}

\section{Generalizations} \label{s:generalizations}

This note, so far, has explained a special case of Lorentzian polynomials with respect to an ample cone, specifically the volume polynomial of matroids. Here, we are able to define the volume polynomial and understand the signature of its Hessian by using the recursive nature of matroids in two ways:
\begin{enumerate}
    \item any interval in the lattice of flats is itself the lattice of flats of another matroid, and
    \item the ample cone obeys $\pi_F(\cK_M)\subseteq \cK_{M_F}$ and $\pi^F(\cK_M)\subseteq \cK_{M^F}$.
\end{enumerate}
For other cases, a more general setting is required. There are two approaches: the theory of Lorentzian polynomials on cones due to Br\"{a}nd\'{e}n and Leake \cite{BL:oncones} and that of Lorentzian fans due to Ross \cite{Ross:Lorentzian}. The theory of Lorentzian polynomials on cones emphasizes the properties of polynomials and their support; while that of Lorentzian fans studies the volume polynomials of tropical fans. 

We first briefly summarize the theory of Lorentzian polynomials on cones.

\begin{definition}
    Let $\cK$ be an open convex cone in $\R^n$. A homogeneous polynomial $f\in \R[x_1,\dots,x_n]$ of degree $d\geq 1$ is call $\cK$-Lorentzian if for all $\v_1,\dots,\v_d\in\cK$,
    \begin{enumerate}
        \item $D_{\v_1}\cdots D_{\v_d}f>0$, and
        \item the Hessian $\cH(D_{\v_3}\cdots D_{\v_d}f)$ has exactly one positive eigenvalue.
    \end{enumerate}
\end{definition}

For $\cK=\R^n_{>0}$, this definition specializes to that of Lorentzian polynomials \cite{BH:Lorentzian}.
This definition immediately implies the log-concavity property seen in Corollary~\ref{c:logconcave}. 

Br\"{a}nd\'{e}n--Leake define a class of {\em hereditary} polynomials closed under taking derivatives with respect to $t_i$ and specialization $t_i\mapsto0$. This allows one to attach a natural ample cone $\cK_f$ to a hereditary polynomial $f$ that shares the recursive nature of the ample cone for matroids. They develop a criterion for determining when a hereditary polynomial is Lorentzian with respect to $\cK_f$, generalizing that in section~\ref{s:polc}. The criterion is in two parts:
\begin{enumerate}
    \item a recursive condition that the derivatives $f$ have Hessians with exactly one positive eigenvalue, and 
    \item a connectedness condition on a simplicial complex built from the support set of $f$.
\end{enumerate}
The connectedness condition translates into the irreducibility condition used in the Lorentzian bootstrap. 

Ross \cite{Ross:Lorentzian} starts with a tropical fan $\Sigma$, i.e., a balanced positively-weighted $d$-dimensional  polyhedral fan in $\R^n$. One looks at $D(\Sigma)\coloneqq \operatorname{PL}(\Sigma)/\operatorname{L}(\Sigma)$, the vector space of piecewise-linear functions on $\Sigma$ modulo linear functions on $\Sigma$. Within it is the cone $\cK(\Sigma)\subset D(\Sigma)$ of strictly convex functions. By tropical intersection theory, one defines a degree map
\[\deg_{\Sigma}\colon D(\Sigma)^{\otimes d}\to \R\]
and asks when for all $D_3,\dots,D_d\in \cK(\Sigma)$, the inner product
\begin{align*}
    D(\Sigma)\times D(\Sigma)&\to \R\\
    (D_1,D_2)&\mapsto \deg(D_1D_2D_3\dots D_d)
\end{align*}
has a unique positive eigenvalue. Ross develops a geometric version of the Lorentzian bootstrap, which guarantees this condition via a recursive condition and a connected hypothesis. The connectedness hypothesis is necessary by a counterexample due to Babaee--Huh \cite{BabaeeHuh}. Ross proves that, in a certain sense, the Lorentzian property of a fan is independent of the fan structure and depends only on the support and weighting. Tihe results of this note can be recovered by specializing this theory to the Bergman fan of matroids \cite{AK:Bergman}.

\bibliographystyle{plain}
\bibliography{references}

\end{document}